\documentclass[11pt, a4paper]{article}
\usepackage{times}
\usepackage{a4wide}
\usepackage{hyperref}
\usepackage[british]{babel}
\usepackage{enumerate, longtable}
\usepackage{amsmath, amscd, amsfonts, amsthm, amssymb, latexsym, comment, stmaryrd, graphicx}
\usepackage[T1]{fontenc}
\usepackage[latin1]{inputenc}

\newtheorem{thm}{Theorem}[section]
\newtheorem{lem}[thm]{Lemma}
\newtheorem{defi}[thm]{Definition}
\newtheorem{rem}[thm]{Remark}

\newtheorem{prop}[thm]{Proposition}
\newtheorem{cor}[thm]{Corollary}

\newtheorem{ass}[thm]{Assumption}

\newtheorem{setup}[thm]{Set-up}

\newcommand{\Hom}{{\rm Hom}}

\newcommand{\GL}{\mathrm{GL}}
\newcommand{\GSp}{\mathrm{GSp}}
\newcommand{\PGSp}{\mathrm{PGSp}}
\newcommand{\PSp}{\mathrm{PSp}}
\newcommand{\Sp}{\mathrm{Sp}}

\newcommand{\proj}{\mathrm{proj}}

\newcommand{\id}{\mathrm{id}}

\DeclareMathOperator{\Image}{im}

\DeclareMathOperator{\Gal}{Gal}

\DeclareMathOperator{\Aut}{Aut}

\DeclareMathOperator{\Frob}{Frob}
\DeclareMathOperator{\Tr}{Tr}

\newcommand{\Ind}{{\rm Ind}}

\newcommand{\PGL}{\mathrm{PGL}}
\newcommand{\PSL}{\mathrm{PSL}}

\newcommand{\calE}{\mathcal{E}}
\newcommand{\calG}{\mathcal{G}}
\newcommand{\calL}{\mathcal{L}}

\newcommand{\calO}{\mathcal{O}}
\newcommand{\cO}{\mathcal{O}}

\newcommand{\fp}{\mathfrak{p}}
\newcommand{\fq}{\mathfrak{q}}

\newcommand{\fL}{\mathfrak{L}}

\newcommand{\FF}{\mathbb{F}}

\newcommand{\NN}{\mathbb{N}}

\newcommand{\QQ}{\mathbb{Q}}

\newcommand{\ZZ}{\mathbb{Z}}

\newcommand{\Qbar}{\overline{\QQ}}

\newcommand{\Fbar}{\overline{\FF}}

\newcommand{\tr}{\mathrm{tr}}

\DeclareMathOperator{\GLReps}{GL-Reps}
\DeclareMathOperator{\GSpReps}{GSp-Reps}

\begin{document}

\selectlanguage{british}

\title{Compatible systems of symplectic Galois representations and the inverse
Galois problem I.\\ Images of projective representations.}
\author{
Sara Arias-de-Reyna\footnote{Universit\'e du Luxembourg,
Facult\'e des Sciences, de la Technologie et de la Communication,
6, rue Richard Coudenhove-Kalergi,
L-1359 Luxembourg, Luxembourg, sara.ariasdereyna@uni.lu},
Luis Dieulefait\footnote{Departament d'\`Algebra i Geometria,
Facultat de Matem\`atiques,
Universitat de Barcelona,
Gran Via de les Corts Catalanes, 585,
08007 Barcelona, Spain, ldieulefait@ub.edu},
Gabor Wiese\footnote{Universit\'e du Luxembourg,
Facult\'e des Sciences, de la Technologie et de la Communication,
6, rue Richard Coudenhove-Kalergi,
L-1359 Luxembourg, Luxembourg, gabor.wiese@uni.lu}}
\maketitle

\begin{abstract}
This article is the first part of a series of three articles about
compatible systems
of symplectic Galois representations and applications to the inverse
Galois problem.

In this first part, we determine the smallest field over which the
projectivisation
of a given symplectic group representation satisfying some natural
conditions can be defined.
The answer only depends on inner twists.
We apply this to the residual representations of a compatible system of
symplectic Galois representations satisfying some mild hypothesis and obtain
precise information on their projective images for almost all members of
the system,
under the assumption of huge residual images, by which we mean that a
symplectic group
of full dimension over the prime field is contained up to conjugation.
Finally, we obtain an application to the inverse Galois problem.

\end{abstract}

\section{Introduction}

This article is the first part of a series of three articles, about compatible systems
of symplectic Galois representations with huge images and applications to the
inverse Galois problem.
The overall aim is, for given $n \in \NN$ even and $d \in \NN$,
to construct a compatible system of Galois representations over~$\QQ$ such that
the projective image of the residual representations at a positive density set
of primes~$\ell$ is $\PGSp_n(\FF_{\ell^d})$ or $\PSp_n(\FF_{\ell^d})$,
thus realising these groups as Galois groups over~$\QQ$.

Three tasks emerge:
\begin{enumerate}[(1)]
\item Given a compatible system of Galois representations, determine the fields of
definition of the projective residual representations.
For $n=2$ and the compatible system arising from a Hecke eigenform $f = \sum_{n \ge 1} a_n q^n$
of level~$N$ and Dirichlet character~$\psi$, these fields are known to be equal to
the residue fields of the number field $\QQ(\frac{a_p^2}{\psi(p)} \;|\; p \nmid N)$
at almost all places.

In this Part~I we produce an analog of this number field for an $n$-dimensional
compatible system of Galois representations. We need to make some irreducibility and
regularity assumptions, which are automatically satisfied for compatible systems
coming from modular forms.

\item Having established the minimal fields of definition for projective representations,
the next task is to determine the actual image. In the case $n=2$ a well-known classical
theorem of Dickson classifies the finite subgroups of $\PGL_2(\Fbar_\ell)$ (up to conjugation) into
`huge' ones (i.e.\ $\PGL_2(\FF_{\ell^d})$ or $\PSL_2(\FF_{\ell^d})$ for some $d \in \NN$),
dihedral ones, subgroups of the upper triangular matrices and three exceptional ones;
in representation theoretic language, the non-huge and non-exceptional images come from
induced or reducible representations.
In Part~II \cite{partII} we supply an analog of this result for $n$-dimensional
symplectic representations under the assumption of the existence of a nontrivial transvection
in the image.

As a general strategy, as in \cite{DiWi}, we will enforce local properties
to the compatible system in order to ensure global ones: one local property takes care of
the nontrivial transvection, another one (that of a maximally induced place, see below)
of absolute irreducibility and in the end huge image and the existence of a suitable subfield
of a cyclotomic field in the projective field of definition.

\item The final task is then to construct a compatible system of Galois representations
satisfying the specified local requirements. This is achieved through the construction
of a suitable automorphic representation in Part~III \cite{partIII},
which is joint work with Sug Woo Shin.
\end{enumerate}

The present Part~I is entirely independent of the other two papers. Part~II is independent as well,
except that in the end the results developed in Part~II are combined with Theorem~\ref{thm:IGP}
from the present Part~I in order to prove that the existence of a compatible
system of Galois representations with certain specified local properties implies
the desired application to the inverse Galois problem.
Part~III relies on some results of Part~I and Part~II.

\subsection*{Statement of the results}

We now describe the main results of the present article in slightly
simplied forms. Although we are mainly
interested in symplectic representations (and compatible systems thereof),
we also prove results for general linear representations along the way;
some of the assumptions needed are different in that case, making the statements
longer; but, we think that including the general linear case is worthwhile.
In this overview we will, however, entirely stick to the symplectic case.

Our first result determines projective fields of definition for Galois representations.
\begin{thm}\label{thm:main1}
Let $L/K$ be a finite Galois extension of $\ell$-adic fields (i.e.\ finite extensions of $\QQ_\ell$)
or of finite fields and $\rho: G \to \GSp_n(L)$
be a representation of a group~$G$ such that
its restriction to the subgroup $I_{[\rho]}$ of~$G$
(defined in Proposition~\ref{prop:K-rho}; $G/I_{[\rho]}$ is finite abelian)
is residually absolutely irreducible.

Then there is an extension $K_{[\rho]}$ of~$K$ inside~$L$, which is explicitly described
only in terms of inner twists, such that
$\rho^\proj: G \xrightarrow{\rho} \GSp_n(L) \to \PGSp_n(L)$ is conjugate to a
representation $G \to \PGSp_n(K_{[\rho]})$,
and $K_{[\rho]}$ is the smallest field with this property.
\end{thm}

This is proved as Theorem~\ref{thm:K-rho} in Section~\ref{sec:inner-group}.
In Section~\ref{sec:compatible} we compute the field $K_{[\overline{\rho}_\lambda]}$
from Theorem~\ref{thm:main1} for almost all residual representations
$\overline{\rho}_\lambda$ in a compatible system of
Galois representations under certain assumptions. The computations are essentially
based on relating the inner twists of the residual representations with
inner twists of the compatible system. Here we present a simplified form
of Theorem~\ref{thm:residual}.

\begin{thm}\label{thm:main2}
Let $\rho_\bullet = (\rho_\lambda)_\lambda$ with $\rho_\lambda: G_\QQ \to \GSp_n(\overline{L}_\lambda)$
(for a Galois number field $L/\QQ$; $\overline{L}_\lambda$ denotes an algebraic closure of
the completion of~$L$ at the place~$\lambda$) be a compatible system as in Definition~\ref{defi:compatible},
consisting of irreducible representations and such that its multiplier is a fixed finite
order character times a fixed power of the cyclotomic character.
Assume a certain condition at places above~$\ell$ (the rational prime below~$\lambda$),
which is for instance satisfied when the $\rho_\lambda$ are de Rham with Hodge-Tate weights
independent of~$\lambda$.

Then there is a number field $K_{\rho_\bullet}$, which is explicitely determined by the inner
twists of the compatible system, such that for all places~$\lambda$ of~$L$ at which $\rho_\lambda$ is residually absolutely irreducible,
except possibly finitely many, the projective
field of definition of the residual representation $\overline{\rho}_\lambda$ is
$\kappa((K_{\rho_\bullet})_\lambda)$, the residue field of $K_{\rho_\bullet}$ at~$\lambda$.
\end{thm}

We can then already determine the residual projective images under the assumption that almost all
of them are `huge' (defined in Definition~\ref{defi:huge}).
This is done in Corollary~\ref{cor:residual}, a version of which we state here.

\begin{thm}\label{thm:main3}
Let $\rho_\bullet = (\rho_\lambda)_\lambda$ be as in Theorem~\ref{thm:main2}.
Assume, moreover, that for all but possibly a density zero set of places~$\lambda$ of~$L$
the residual representation $\overline{\rho}_\lambda$ has huge image.

Then for all places $\lambda$ of~$L$ with the possible exception of a density zero set,
the image of $\overline{\rho}_\lambda^\proj$
is $\PGSp_n(\kappa((K_{\rho_\bullet})_\lambda))$ or $\PSp_n(\kappa((K_{\rho_\bullet})_\lambda))$.
\end{thm}

Under these assumptions together with that of the existence of a maximally induced
place of a certain order (see Definition~\ref{defi:maximally-induced},
this goes back to~\cite{KLS1} in the context of $(n,p)$-groups)
we can derive the following application to the inverse Galois problem
(proved as Theorem~\ref{thm:IGP}).

\begin{thm}\label{thm:main4}
Let $\rho_\bullet = (\rho_\lambda)_\lambda$ be as in Theorem~\ref{thm:main3}.
Assume also that there is a place $\fq \in S$
such that $\rho_\bullet$ is maximally induced at~$\fq$ of order~$p$.

Then for any $d \mid \frac{p-1}{n}$ there exists a set~$\calL_d$ of rational primes~$\ell$
of positive density such that for all $\ell \in \calL_d$ there is a place $\lambda$ of~$L$
above~$\ell$ satisfying that the image of $\overline{\rho}_\lambda^\proj$
is $\PGSp_n(\FF_{\ell^d})$ or $\PSp_n(\FF_{\ell^d})$.
\end{thm}

The maximally induced place of order~$p$ serves to produce a certain element
$\xi_p \in \QQ(\zeta_p)$ of degree $\frac{p-1}{n}$ in~$K_{\rho_\bullet}$,
generalising the approach of~\cite{DiWi}. This element
guarantees the existence of a subextension of~$K_{\rho}$ that is cyclic degree~$d$,
from which the existence of a positive density set of primes of residue degree~$d$ can be derived.

\subsection*{Acknowledgements}

S.~A.-d.-R.\ worked on this article as a fellow of the Alexander-von-Humboldt foundation.
She thanks the Universit\'e du Luxembourg for its hospitality during a long term visit.
She was also partially supported by the project MTM2012-33830 of the Ministerio de Econom\'ia y Competitividad of Spain.  
L.~V.~D. was supported by the project MTM2012-33830 of the Ministerio de Econom\'ia y Competitividad of Spain and by an ICREA Academia Research Prize.
G.~W.\ was partially supported by the Sonderforschungsbereich TRR~45,
the DFG priority program~1489 and by the Universit\'e du Luxembourg.

\section{Inner twists of a group representation}\label{sec:inner-group}

In this section we determine under some mild conditions the smallest field over which
a member of the equivalence class of a (projective, symplectic) representation is defined.
The main results are Theorem~\ref{thm:K-rho} and Proposition~\ref{prop:E-rho}.
This part owes many of its ideas to Ribet's paper~\cite{Ri}, including
Papier's theorem cited in that article, as well as to other papers of Ribet.
A conceptual framework for some of those ideas is provided here as well as a generalisation.

We take some time and space to introduce the objects we work with in detail. A guiding
example the reader may have in mind is that of a newform $f=\sum_{n\ge 1} a_n q^n$ with
coefficient field $\QQ_f = \QQ(a_n \;|\; n \in \NN)$. For a prime~$\ell$ take $K= \QQ_\ell$
and $L$ the Galois closure over~$\QQ_\ell$ of the completion of $\QQ_f$ at some place above~$\ell$
(or $L=\Qbar_\ell$).
The $\rho$ we will be interested in are the $\ell$-adic Galois representation $\rho_{f}$
and its mod~$\ell$ reduction~$\overline{\rho}_{f}$ attached to $f$ with respect to the chosen place.

We are interested in $\ell$-adic fields (i.e.\ finite extensions of $\QQ_\ell$) as well as finite fields. Since the reasonings are very similar in both cases, we will treat them simultaneously. In what follows, {\bf when we write that a representation is ``(residually) absolutely irreducible'' we mean that it is ``residually absolutely irreducible'' in
the case of $\ell$-adic fields, and ``absolutely irreducible'' in the case of finite fields.}
Whenever we consider a morphism between two topological spaces we assume it to be continuous. In particular, representations and  characters will always be continuous.

\begin{setup}\label{setup}

\begin{itemize} 
\itemsep=0cm plus 0pt minus 0pt
\parsep=0.0cm

\item $K$ an $\ell$-adic field with the $\ell$-adic topology or a finite field with the discrete topology.

\item $L/K$ a Galois extension, $\Gamma := \Gal(L/K)$, endowed with the Krull topology.

\item $G$ a finite or a compact topological group. In later sections, $G$ will be a Galois group.

\item $\calE = \{ \epsilon: G \to L^\times \textnormal{character }\}$.

Note that the assumptions imply that $\Image(\epsilon)$ lies in a finite extension of $K$. Note also that $\Gamma$ acts on $\calE$ from the left by composition:
${}^\gamma \epsilon := \gamma \circ \epsilon$.

\item Form the semi-direct product $\calG := \calE \rtimes \Gamma$ for the above action.
Since we consider the $\Gamma$-entry more important than the $\calE$-entry, we denote
elements of $\calG$ as $(\gamma,\epsilon)$ with $\gamma \in \Gamma$ and $\epsilon \in \calE$
(instead of the other way around). We concretely have:
$$ (\gamma_1,\epsilon_1) \cdot (\gamma_2,\epsilon_2)
   := (\gamma_1\gamma_2 , ({}^{\gamma_1}\epsilon_2) \epsilon_1)
\textnormal{ and }
(\gamma,\epsilon)^{-1} = (\gamma^{-1},{}^{\gamma^{-1}}(\epsilon^{-1})).$$
Consequently, we have the exact sequence:
$$ 1 \to \calE \xrightarrow{\iota: \epsilon \mapsto (1,\epsilon)} \calG
               \xrightarrow{\pi: (\gamma,\epsilon) \mapsto \gamma} \Gamma \to 1.$$
This sequence is split in the obvious way.
Note that the other projection
$$c: \calG \to \calE, \;\;\; (\gamma,\epsilon) \mapsto \epsilon$$
is a $1$-cocycle, when letting $\calG$ act on $\calE$ through its quotient~$\Gamma$.

\item Let $n \in \NN$. The equivalence class of a representation
$\rho: G \to \GL_n(L)$ (for conjugation by an element of $\GL_n(L)$)
is denoted by~$[\rho]$ and equivalence of representations is denoted by~$\sim$.
We denote the set of equivalence classes by $\GLReps_n(G, L/K)$. Note $\rho$
is defined over a finite extension $\tilde{K}/K$ inside~$L$, that is
$\rho: G \to \GL_n(\tilde{K}) \xrightarrow{\textnormal{nat.emb.}} \GL_n(L)$ (This is a standard fact: the reader can find a proof in e.g.~\cite{Skinner08}, page 244).

\item Let $n \in \NN$ be even. The symplectic equivalence class of a representation
$\rho: G \to \GSp_n(L)$ (for conjugation by an element of $\GSp_n(L)$)
is also denoted by~$[\rho]$ and also symplectic equivalence of representations is denoted by~$\sim$.

We denote the set of equivalence classes by $\GSpReps_n(G,L/K)$. Note that
$\rho$ is defined over a finite extension $\tilde{K}/K$ inside~$L$, that is
$\rho: G \to \GSp_n(\tilde{K}) \xrightarrow{\textnormal{nat.emb.}} \GSp_n(L)$.
We make the convention that equivalence of symplectic representations always means
symplectic equivalence.
Note that the embedding $\GSp_n(L) \hookrightarrow \GL_n(L)$ allows us to see
elements of $\GSpReps_n(G,L/K)$ as elements of $\GLReps_n(G, L/K)$.
We will call $[\rho]$ (absolutely) irreducible, if (any member of) it is (absolutely) irreducible as a linear
representation.

\item Note that $\calG$ acts on $\GLReps_n(G, L/K)$ from the left by the following formula:
$$ (\gamma,\epsilon). [\rho] := [({}^\gamma \rho) \otimes_{L} \epsilon^{-1}],$$
where $\gamma \in \Gamma$, $\epsilon \in \calE$ and $[\rho] \in \GLReps_n(G, L/K)$.
In the same way, $\calG$ acts on the set $\GSpReps_n(G,L/K)$ from the left.

\item  Let $n \in \NN$ and $r_1,r_2: G \to \PGL_n(L)$ be
projective representations. We call $r_1$ and $r_2$ equivalent (also denoted $r_1 \sim r_2$)
if they are conjugate by the class (modulo scalars) of a matrix in $\GL_n(L)$.
The equivalence class of $r_1$ is also denoted $[r_1]$.
For $\rho: G\rightarrow \GL_n(L)$ we denote by $\rho^\proj$ the composition of $\rho$ with the natural projection $\GL_n(L)\rightarrow \PGL_n(L)$.

\item  Let $n \in \NN$ be even and $r_1,r_2: G \to \PGSp_n(L)$ be
projective symplectic representations. We call $r_1$ and $r_2$ equivalent (also denoted $r_1 \sim r_2$)
if they are conjugate by the class (modulo scalars) of a matrix in $\GSp_n(L)$.
The equivalence class of $r_1$ is also denoted $[r_1]$.
For $\rho: G\rightarrow \GSp_n(L)$ we denote by $\rho^\proj$ the composition of $\rho$ with the natural projection $\GSp_n(L)\rightarrow \PGSp_n(L)$.
\end{itemize}
\end{setup}

\begin{rem}\label{rem:rep_integrality}
Let $K$ be an $\ell$-adic field and $\calO_K$ its valuation ring.
\begin{enumerate}[(a)]
\item\label{rem:rep_integrality:a} 
It is well-known that any representation $\rho:G \to \GL_n(K)$ is conjugate in $\GL_n(K)$
to a representation $G \to \GL_n(\calO_K)$ (see e.g.\ \cite{serre-abelian}, I-2).
We have already used this implicitly when speaking about residual representations.

\item\label{rem:rep_integrality:b}
In the symplectic case we assume that $\rho:G \to \GSp_n(K)$ is residually
irreducible. One can show that then $\rho$ is conjugate in $\GSp_n(K)$
to a representation $G \to \GSp_n(\calO_K)$.

Indeed, if $\calL$ is a lattice provided by~(\ref{rem:rep_integrality:a}) and the standard
symplectic pairing is scaled such that its restriction to $\calL \times \calL$ is
surjective onto~$\calO_K$, then the residual irreducibility together with Nakayama's lemma
implies that the $\calO_K[G]$-homomorphism
$\calL \xrightarrow{v \mapsto (w \mapsto \langle w,v \rangle)} \Hom_{\calO_K}(\calL,\calO_K)$,
where the $G$-action on the right hand side is twisted by the multiplier of~$\rho$,
is an isomorphism. Thus, $\calL$ admits a symplectic $\calO_K$-basis and the
corresponding base change matrix is general symplectic.
\end{enumerate}
\end{rem}

For the rest of this section we assume that we are in Set-up~\ref{setup}. We now make the main definition of this section, which generalises standard
definitions for elliptic curves and modular forms.

\begin{defi}\label{defi:main}
Let $[\rho] \in \GLReps_n(G, L/K)$ or $[\rho] \in \GSpReps_n(G,L/K)$.
\begin{itemize}
\itemsep=0cm plus 0pt minus 0pt
\parsep=0.0cm

\item Define $\calG_{[\rho]}$ to be the stabiliser group of~$[\rho]$ in~$\calG$
(for the $\calG$-action on $\GLReps_n(G, L/K)$ or on $\GSpReps_n(G,L/K)$).
Its elements are called {\em inner twists of~$[\rho]$}.
Explicitly, let $(\gamma,\epsilon) \in \calG$. Then $(\gamma,\epsilon)$ is an inner twist of~$[\rho]$
if and only if $[\rho] = [({}^\gamma \rho) \otimes_{L} \epsilon^{-1}]$,
which is the case if and only if $[{}^\gamma \rho] = [\rho \otimes_{L} \epsilon]$.

\item Define the groups $\Gamma_{[\rho]} := \pi(\calG_{[\rho]}) \subseteq \Gamma$ and
$\calE_{[\rho]} := \iota^{-1}(\calG_{[\rho]}) = \iota^{-1}(\ker(\pi|_{\calG_{[\rho]}}))$.

\item We say that $[\rho]$ has {\em no complex multiplication} if $\calE_{[\rho]} = \{1\}$,
i.e.\ if and only if the projection $\pi:\calG_{[\rho]} \rightarrow \Gamma_{[\rho]}$ is an isomorphism.

\item Define the field $K_{[\rho]} := L^{\Gamma_{[\rho]}}$. It is called
the {\em projective field of definition of~$[\rho]$}
(see Theorem~\ref{thm:K-rho} for a justification of the terminology).

\item Define the group $\Delta_{[\rho]} := \{ \gamma \in \Gamma_{[\rho]} \;|\; (\gamma,1) \in \calG_{[\rho]} \}$.

\item Define the field $E_{[\rho]} := L^{\Delta_{[\rho]}}$. It is called
the {\em field of definition of~$[\rho]$} (see Proposition~\ref{prop:E-rho}~(\ref{prop:E-rho:b})
for a justification of the terminology).
\end{itemize}
\end{defi}

In our guiding example of a newform~$f$, $\rho_f$ has CM or a nontrivial inner twist
if and only if $f$ does.

\begin{lem}\label{lem:GSp_versus_GL}
\begin{enumerate}[(a)]
\item\label{prop:rep-theory:a} Let $[\rho] \in \GLReps_n(G, L/K)$ or $[\rho] \in \GSpReps_n(G,L/K)$
be absolutely irreducible. Then any $M \in \GL_n(L)$ commuting with all
$\rho(g)$ for $g \in G$ is a scalar matrix.
\item\label{lem:GSp_versus_GL:b} Let $\rho:G\rightarrow \GSp_n(L)$ be absolutely irreducible without complex multiplication, let $M\in GL_n(L)$ and assume that the image of $M^{-1}\rho M$ lies in $\GSp_n(L)$. Then $M\in \GSp_n(L)$, thus  $M^{-1}\rho M$ belongs to the symplectic equivalence class of $\rho$.
\end{enumerate}
\end{lem}

\begin{proof}
(\ref{prop:rep-theory:a}) This is a form of Schur's Lemma. See e.g.\ \cite{CurRei_Rep}, $\S$ 27 of Chapter IV, Lemma (27.3).

(\ref{lem:GSp_versus_GL:b}) Call $\rho' = M^{-1} \rho M$.
The symplecticity of $\rho$ and $\rho'$ implies for $g \in G$:
$$\rho(g)^\tr J \rho(g) = m(\rho(g)) J \textnormal{ and }
(M^{-1} \rho(g) M)^\tr J(M^{-1} \rho(g) M) = m(\rho'(g)) J,$$
where $m$ is the multiplier and $J$ is the Gram matrix of the standard
symplectic pairing. Combining these two equations, yields
$$ (M J^{-1} M^\tr J)^{-1} \rho(g) (M J^{-1} M^\tr J) = \frac{m(\rho'(g))}{m(\rho(g))} \rho(g).$$
Note that this equation implies that the character $\frac{m(\rho'(g))}{m(\rho(g))}$
lies in~$\calE_{[\rho]}$. As we are assuming that $\rho$ has no complex multiplication,
this character is trivial.
Using (\ref{prop:rep-theory:a}), we get that $M J^{-1} M^\tr J$ is a scalar matrix $\lambda \id_n$,
showing that $M \in \GSp_n(L)$, as claimed.
\end{proof}

One can find examples showing that the conditions in
Lemma \ref{lem:GSp_versus_GL}~(\ref{lem:GSp_versus_GL:b}) are necessary.
On our way to develop the main results of this section,
we first include a proposition summarising some representation theory to be used in the sequel.

\begin{prop}\label{prop:rep-theory}
\begin{enumerate}[(a)]

\item\label{prop:rep-theory:b} Let $[\rho], [\rho'] \in \GLReps_n(G, L/K)$ be absolutely irreducible.
If for all $g \in G$ the equality $\Tr(\rho(g)) = \Tr(\rho(g'))$ holds,
then $[\rho]=[\rho']$.
\item\label{prop:rep-theory:c} Let $[\rho], [\rho'] \in \GSpReps_n(G,L/K)$ be absolutely irreducible without
complex multiplication.
If $\Tr(\rho(g)) = \Tr(\rho'(g))$ for all $g \in G$,
then $[\rho]=[\rho']$.
\item\label{prop:rep-theory:d} Let $[\rho] \in \GLReps_n(G, L/K)$ be (residually) absolutely irreducible and let $K \subseteq \tilde{K} \subseteq L$
be a field extension such that $\Tr(\rho(g)) \in \tilde{K}$ for all $g \in G$.
Then there is $M \in \GL_n(L)$ such that $M^{-1} \rho(g) M \in \GL_n(\tilde{K})$ for
all $g\in G$. So, the equivalence class $[\rho]$ has a representative with target
in $\GL_n(\tilde{K})$.
\item\label{prop:rep-theory:e} Let $[\rho] \in \GSpReps_n(G,L/K)$ be (residually) absolutely irreducible
without complex multiplication and let $K \subseteq \tilde{K} \subseteq L$
be a field extension such that $\Tr(\rho(g)) \in \tilde{K}$ and $m(g) \in \tilde{K}$
(multiplier map) for all $g \in G$.
Then there is $M \in \GSp_n(L)$ such that $M^{-1} \rho(g) M \in \GSp_n(\tilde{K})$ for
all $g\in G$. So, the equivalence class $[\rho]$ has a representative with target
in $\GSp_n(\tilde{K})$.
\end{enumerate}
\end{prop}

\begin{proof}

(\ref{prop:rep-theory:b}) This follows from the Theorem of Brauer-Nesbitt, cf. \cite{CurRei_Rep}, $\S$30 of Chapter V, Theorem (3.15). 

(\ref{prop:rep-theory:c}) Applying (\ref{prop:rep-theory:b}),
we obtain that $\rho$ and $\rho'$ are equivalent as $\GL_n$-representations,
meaning that there is some matrix $M \in \GL_n(L)$ such that $\rho' = M^{-1} \rho M$.
Lemma \ref{lem:GSp_versus_GL}-(\ref{lem:GSp_versus_GL:b}) shows that $M \in \GSp_n(L)$, as claimed.

(\ref{prop:rep-theory:d}) For $\ell$-adic fields, this is Th\'eor\`eme~2 of~\cite{Ca}; see also \cite{Ma}, p.~255.
For finite fields the result follows from the triviality of the Brauer group.

(\ref{prop:rep-theory:e}) Applying (\ref{prop:rep-theory:d}) we obtain a matrix $M  \in \GL_n(L)$ such that
$M^{-1} \rho(g) M \in \GL_n(\tilde{K})$ for all $g \in G$. We want to show that there is a matrix
$N \in \GL_n(\tilde{K})$ such that $N^{-1} M^{-1} \rho(g) M N \in \GSp_n(\tilde{K})$
for all $g \in G$.
Again denote $J$ the Gram matrix of the standard symplectic pairing.
Then all $M^{-1} \rho(g) M$ respect the symplectic pairing with Gram
matrix $I := M^\tr J M$ up to the same multiplier $m(\rho(g))$, i.e.\
$$ (M^{-1} \rho(g) M)^\tr I (M^{-1} \rho(g) M) = m(\rho(g)) I$$
for all $g \in G$.

We claim that $I \in \GL_n(\tilde{K})$. Note that conjugating $M^{-1} \rho(g) M$
by~$I$ yields an element in $\GL_n(\tilde{K})$,
using here the assumption $m(\rho(g)) \in \tilde{K}$.
The absolute irreducibility of~$M^{-1} \rho M$ now implies that
conjugating by~$I$ preserves $M_n(\tilde{K})$ (see e.g.\ \cite{Ma}, p.~252).
The Skolem-Noether theorem (Corollary 3.63 of~\cite{CR})
consequently shows $I \in \GL_n(\tilde{K})$.

This now means that $I$ defines a symplectic pairing on $\tilde{K}^n$
(the representation space of~$M^{-1} \rho M$), which is respected by $M^{-1} \rho M$
(up to the multiplier).
The claimed existence of the matrix $N$ is now just the fact that in a symplectic
vector space a basis can be chosen such that the symplectic pairing has the standard form,
i.e.\ is given by~$J$. We conclude, using~(\ref{prop:rep-theory:c}), that $(MN)^{-1} \rho (MN)$
is in the equivalence class~$[\rho]$.
\end{proof}

We now include a word on topology.

\begin{lem}\label{lem:topology}
Let $[\rho] \in \GLReps_n(G, L/K)$ or $[\rho] \in \GSpReps_n(G,L/K)$.
Then $\Gamma_{[\rho]}$ and $\Delta_{[\rho]}$ are open subgroups of~$\Gamma$ (and thus of finite index).
\end{lem}

\begin{proof}
There is an open normal subgroup
$\widetilde{\Gamma}_{\rho}$ of~$\Gamma$ acting trivially on~$\rho$.
Then we have inclusions
$$ \widetilde{\Gamma}_{\rho} \subseteq \Delta_{[\rho]} \subseteq \Gamma_{[\rho]} \subseteq \Gamma,$$
all of which are of finite index.
Fix a system of representatives $\gamma_1,\dots,\gamma_m$ of
$\Gamma/\widetilde{\Gamma}_{\rho}$. Then
$$ \calE_i := \{\epsilon \in \calE \;|\; {}^{\gamma_i} \rho \sim \rho \otimes \epsilon\}
= \begin{cases}
\emptyset &\textnormal{ or} \\ \epsilon_i \calE_{[\rho]}& \textnormal{ for some } \epsilon_i \in \calE.
\end{cases}$$
It follows that
$\Gamma_{[\rho]} = \bigcup_{i=1, \calE_i \neq \emptyset}^m \gamma_i \widetilde{\Gamma}_{\rho}$
and
$\Delta_{[\rho]} = \bigcup_{i=1, \calE_i = \calE_{[\rho]}}^m \gamma_i \widetilde{\Gamma}_{\rho}$, thus they
are open subsets of~$\Gamma$.
\end{proof}

The following lemma is very useful in our applications of the above theory to compatible systems
of Galois representations, where the Frobenius traces lie in some global field.

\begin{lem}\label{lem:G-on-traces}
Let $[\rho] \in \GLReps_n(G, L/K)$ or $[\rho] \in \GSpReps_n(G,L/K)$ (in this case also
assume that $\rho$ has no complex multiplication) be absolutely irreducible.
Then $\calG_{[\rho]}$ is equal to the set
$$\{ (\gamma,\epsilon) \in \calG \;|\; \gamma(\Tr(\rho(g))) = \Tr(\rho(g)) \epsilon(g) \;\forall g \in G\}.$$
\end{lem}

\begin{proof}
Proposition~\ref{prop:rep-theory} (\ref{prop:rep-theory:b}) and~(\ref{prop:rep-theory:c}).
\end{proof}

We continue with a useful observation on the values of $\epsilon$ occuring for inner twists.

\begin{lem}\label{lem:order}
Let $[\rho] \in \GLReps_n(G, L/K)$ or $[\rho] \in \GSpReps_n(G,L/K)$.
Furthermore, let $\chi: G \to K^\times$ be any character
and let $\psi: G \to L^\times$ be a character of finite order.
Let $\epsilon \in \calE$ occur in some $(\gamma,\epsilon) \in \calG_{[\rho]}$.
\begin{enumerate}[(a)]
\item\label{lem:order:a} If $\det(\rho) = \psi \chi$, then $\epsilon^n = \frac{{}^\gamma \psi}{\psi}$ and
the values of $\epsilon^n$ are contained in the field generated over~$K$ by the values of~$\psi$.
\item\label{lem:order:b} If $[\rho]$ is symplectic with multiplier map
$m(\rho) = \psi \chi$, then $\epsilon^2 = \frac{{}^\gamma \psi}{\psi}$
and the values of $\epsilon$ are contained in the field generated over~$K$ by the values of~$\psi$.
\end{enumerate}
\end{lem}

\begin{proof}
Let $(\gamma,\epsilon) \in \calG_{[\rho]}$. We get:

(\ref{lem:order:a}) $\gamma (\det(\rho)) = \chi \gamma(\psi) = \psi \chi \epsilon^n$,
implying $\epsilon^n = \frac{{}^\gamma \psi}{\psi}$. The second assertion is now clear.

(\ref{lem:order:b}) The same calculation with the multiplier map yields $\epsilon^2
= \frac{{}^\gamma \psi}{\psi}$. Consequently, the order of $\epsilon$ divides~$2m$,
where $m$ is the order of~$\psi$.
Let $\zeta$ be a primitive $m$-th root of unity, $g\in G$ and $b, c\in \ZZ$
such that $\psi(g)=\zeta^b$, $\gamma(\zeta)=\zeta^c$. Then
$\epsilon(g)^2=\frac{\zeta^{bc}}{\zeta^{b}}=\zeta^{b(c-1)}$.
If $m$ is odd, there exists $d\in \ZZ$ such that $\zeta=(\zeta^2)^d$, thus 
$\epsilon(g)^2=\zeta^{2bd(c-1)}$ and $\epsilon(g)\in K(\zeta)$.
Assume now that $m$ is even. Then $c-1$ is also even (otherwise $c$ would be even,
and $\zeta^c$ would not be a primitive $m$-th root of unity).
Hence $\epsilon(g)^2$ has order dividing~$m/2$,
so that $\epsilon(g)$ has order dividing~$m$; thus the values of $\epsilon$ are a subset of
the values of~$\psi$.
\end{proof}

We now study projective representations.

\begin{lem}\label{lem:proj-char}
\begin{enumerate}[(a)]
\item\label{lem:proj-char:a} For $\rho: G \to \GL_n(L)$ and $\epsilon \in \calE$ one has
$\rho^\proj = (\rho \otimes \epsilon)^\proj$.

\item\label{lem:proj-char:b} For $[\rho] \in \GLReps_n(G, L/K)$ or $[\rho] \in \GSpReps_n(G,L/K)$ and $\epsilon \in \calE$ one has
$\rho^\proj \sim (\rho \otimes \epsilon)^\proj$.

\item\label{lem:proj-char:c} Let $[\rho_1],[\rho_2] \in \GLReps_n(G, L/K)$ or $[\rho_1],[\rho_2] \in \GSpReps_n(G,L/K)$
such that $\rho_1^\proj \sim \rho_2^\proj$.
Then there is $\epsilon \in \calE$ such that $[\rho_1] = [\rho_2 \otimes \epsilon]$.

\end{enumerate}
\end{lem}

\begin{proof}
(\ref{lem:proj-char:a}) and (\ref{lem:proj-char:b}) are clear.

(\ref{lem:proj-char:c}) Denote by $s$ the injective homomorphism $L^\times \to \GL_n(L)$
which sends $x$ to the scalar matrix $x \cdot \id_n$.
Let $M$ be a matrix in $\GL_n(L)$ (or in $\GSp_n(L)$)
such that $\rho_1^\proj = M \rho_2^\proj M^{-1}$.
Define the character $\epsilon: G \to L^\times$ by the formula
$$ \epsilon(g) := s^{-1}\big(\rho_1(g) (M\rho_2(g)M^{-1})^{-1}\big).$$
Note that $\epsilon$ is well-defined, is multiplicative
and that it obviously satisfies the requirements.
\end{proof}

For the next lemma note that $({}^\gamma \rho)^\proj = {}^\gamma ( \rho^\proj )$.

\begin{lem}\label{lem:proj}
Let $[\rho] \in \GLReps_n(G, L/K)$ or $[\rho] \in \GSpReps_n(G,L/K)$.
\begin{enumerate}[(a)]
\item\label{lem:proj:a} $\Gamma_{[\rho]} = \{ \gamma \in \Gamma \;|\; {}^\gamma \rho^\proj \sim \rho^\proj \}$.
\item\label{lem:proj:b} If $\rho^\proj$ factors as $G \to \PGL_n(\tilde{K}) \xrightarrow{\textnormal{nat.~emb.}} \PGL_n(L)$
or as $G \to \PGSp_n(\tilde{K}) \xrightarrow{\textnormal{nat.~emb.}} \PGSp_n(L)$
for some field $K \subseteq \tilde{K} \subseteq L$,
then $K_{[\rho]} \subseteq \tilde{K}$.
\end{enumerate}
\end{lem}

\begin{proof}
(\ref{lem:proj:a}) Let $\gamma \in \Gamma_{[\rho]}$. Then there is $\epsilon$ such that $(\gamma,\epsilon) \in \calG_{[\rho]}$,
i.e.\ ${}^\gamma \rho \sim \rho \otimes \epsilon$. Hence, ${}^\gamma \rho^\proj \sim \rho^\proj$
by Lemma~\ref{lem:proj-char}.
Conversely, if $\gamma \in \Gamma$ satisfying ${}^\gamma \rho^\proj \sim \rho^\proj$, then by
Lemma~\ref{lem:proj-char}, there is $\epsilon$ such that $(\gamma,\epsilon) \in \calG_{[\rho]}$,
whence $\gamma \in \Gamma_{[\rho]}$.

(\ref{lem:proj:b}) Let $U \subseteq \Gamma$ the closed subgroup such that $\tilde{K} = L^U$.
Then for all $\gamma \in U$ we have ${}^\gamma \rho^\proj = \rho^\proj$. Hence, $\gamma \in \Gamma_{[\rho]}$
by~(\ref{lem:proj:a}). Consequently, $U \subseteq \Gamma_{[\rho]}$, whence $K_{[\rho]} \subseteq \tilde{K}$.
\end{proof}

The lemma shows that any field over which $\rho^\proj$ can be defined contains $K_{[\rho]}$.
Our aim will be to show that under suitable assumptions $\rho^\proj$ can be defined over $K_{[\rho]}$.

\begin{lem}\label{lem:charpoly-coeff}
Let $[\rho] \in \GLReps_n(G, L/K)$ or $[\rho] \in \GSpReps_n(G,L/K)$.
Let $g \in G$ and let $X^n + \sum_{i=1}^n a_i(g) X^{n-i}$
be the characteristic polynomial of $\rho(g)$. Let $(\gamma,\epsilon) \in \calG_{[\rho]}$.
Then one has the equations
$$ \gamma(a_i(g)) = \epsilon(g)^i \cdot a_i(g)$$
for all~$i=1,\dots,n$.
\end{lem}

\begin{proof}
This follows by comparing the characteristic polynomials of $[\rho]$
and $(\gamma,\epsilon).[\rho]$.
\end{proof}

\begin{cor}\label{cor:charpoly-coeff}
Let $[\rho] \in \GLReps_n(G, L/K)$ or $[\rho] \in \GSpReps_n(G,L/K)$ and $\epsilon \in \calE_{[\rho]}$.
\begin{enumerate}[(a)]
\item\label{cor:charpoly-coeff:a} The order of $\epsilon$ divides~$n$.
\item\label{cor:charpoly-coeff:b} Let $i \in \{1,\dots,n\}$ and $g \in G$. If $a_i(g) \neq 0$, then $\epsilon^i(g) = 1$.
\end{enumerate}
\end{cor}

\begin{proof}
(\ref{cor:charpoly-coeff:a}) The coefficient $a_n(g)$ of the characteristic polynomial of $\rho(g)$ as in
Lemma~\ref{lem:charpoly-coeff} is the determinant of $\rho(g)$, which is a unit.
As $a_n(g) = \epsilon(g)^n a_n(g)$ for all~$g$, we conclude $\epsilon(g)^n=1$.

(\ref{cor:charpoly-coeff:b}) is immediate.
\end{proof}

\begin{lem}\label{lem:cm}
For $[\rho] \in \GLReps_n(G, L/K)$ or $[\rho] \in \GSpReps_n(G,L/K)$, consider
$H_{[\rho]} := \bigcap_{\epsilon \in \calE_{[\rho]}} \ker(\epsilon)$.
It is a closed normal subgroup of~$G$ with abelian quotient of exponent dividing~$n$.

If $\rho$ has complex multiplication, then the restriction of $\rho$ to $H_{[\rho]}$ is
not absolutely irreducible.
\end{lem}

\begin{proof}
Let $\epsilon \in \calE_{[\rho]}$ and let $M$ be a (symplectic) matrix
such that $\rho \otimes \epsilon = M \rho M^{-1}$.
If the restriction of $\rho$ to $H_{[\rho]}$ is absolutely irreducible, then restricting to~$H_{[\rho]}$ gives
$\rho |_{H_{[\rho]}}= M\rho M^{-1}|_{H_{[\rho]}}$. By absolute irreduciblity,
any matrix that commutes with all of~$\rho(H_{[\rho]})$ is scalar,
whence $\rho \otimes \epsilon = \rho$. Consequently, $\epsilon$ is the trivial character.
\end{proof}

\begin{prop}\label{prop:K-rho}
For $[\rho] \in \GLReps_n(G, L/K)$ or $[\rho] \in \GSpReps_n(G,L/K)$,
define the subgroup $I_{[\rho]} := \bigcap_{\epsilon \in \calE \textnormal{ s.t. } \exists
(\gamma,\epsilon) \in \calG_{[\rho]}} \ker(\epsilon)$.
Assume that the restriction of $\rho$ to $I_{[\rho]}$ is (residually) absolutely irreducible
(in particular, by Lemma~\ref{lem:cm} this implies that $[\rho]$ has no complex multiplication).
Then the equivalence class $[\rho]$ contains a member $\rho'$ satisfying:
\begin{itemize}
\item Let $g \in G$. Then $\rho'(g) \in \GL_n(K_{[\rho]}) \Leftrightarrow g \in I_{[\rho]}$.
\item ${}^\gamma \rho' = \rho' \otimes \epsilon$ for all $(\gamma,\epsilon) \in \calG_{[\rho]}$.
\end{itemize}
\end{prop}

\begin{proof}
Due to Lemma~\ref{lem:charpoly-coeff}, the character of $\rho$ restricted
to $I_{[\rho]}$ takes values in $K_{[\rho]}$.
If $\rho$ is symplectic, the formula $\epsilon^2 = \frac{{}^\gamma (m\circ \rho)}{m \circ \rho}$
for all $(\gamma,\epsilon) \in \calG_{[\rho]}$
(see also Lemma~\ref{lem:order}~(\ref{lem:order:b})) shows that
the multiplier of~$\rho$ restricted to $I_{[\rho]}$ takes values in~$K_{[\rho]}$.
Due to the (residual) irreduciblity assumption, $\rho|_{I_{[\rho]}}$ is thus conjugate by some
(symplectic) matrix $M$ to a representation defined over $K_{[\rho]}$,
see Proposition~\ref{prop:rep-theory} (\ref{prop:rep-theory:d}) and~(\ref{prop:rep-theory:e}).
Thus $\rho' := M \rho M^{-1}$ is a representation of~$G$ such that
its restriction to $I_{[\rho]}$ takes its values in the (symplectic) matrices with entries in~$K_{[\rho]}$.

We start by showing the second property. For this, let $(\gamma,\epsilon) \in \calG_{[\rho]}$.
Since $[{}^\gamma \rho] = [\rho \otimes \epsilon]$, there is a (symplectic) matrix $N$ such that
${}^\gamma \rho' = N (\rho' \otimes \epsilon) N^{-1}$. Restricting to $I_{[\rho]}$ and using that this
restriction is absolutely irreducible, the matrix~$N$ has to be scalar
(see Proposition~\ref{lem:GSp_versus_GL}~(\ref{prop:rep-theory:a})),
showing ${}^\gamma \rho' = \rho' \otimes \epsilon$.

Let now $g \in G \setminus I_{[\rho]}$. Then there is
a $(\gamma,\epsilon) \in \calG_{[\rho]}$ such that $\epsilon(g) \neq 1$. Hence,
$\gamma(\rho'(g)) = \rho'(g) \epsilon(g) \neq \rho'(g)$, showing $\rho'(g) \not\in \GL_n(K_{[\rho]})$.
\end{proof}

Our first main theorem can now be obtained as an application of Hilbert's Satz~90.

\begin{thm}\label{thm:K-rho}
Let $[\rho] \in \GLReps_n(G, L/K)$ or $[\rho] \in \GSpReps_n(G,L/K)$
such that its restriction to $I_{[\rho]}$ (see Proposition~\ref{prop:K-rho} for its definition)
is (residually) absolutely irreducible.
Then the equivalence class of $\rho^\proj$ has a member that factors through $\PGL_n(K_{[\rho]})$
or $\PGSp_n(K_{[\rho]})$, respectively.
Combining with Lemma~\ref{lem:proj}, $K_{[\rho]}$ is hence the smallest subfield of~$L$
with this property.
\end{thm}

\begin{proof}
By Proposition~\ref{prop:K-rho}, we may and do assume that
${}^\gamma \rho = \rho \otimes \epsilon$ for all $(\gamma,\epsilon) \in \calG_{[\rho]}$.

Now let $g \in G$. Consider the continuous $1$-cocycle $c_g:\Gamma_{[\rho]}\rightarrow L^{\times}$ defined by 
$c_g(\gamma)=s^{-1}({}^{\gamma}\rho(g)\rho(g)^{-1})$, where $s:L^{\times}\rightarrow \GL_n(L)$ denotes the morphism which sends $x$ to the scalar matrix $x\cdot \id_n$.
By Hilbert's Satz 90, $H^1(\Gamma_{[\rho]}, L^\times)$
is the trivial module. Hence, $c_g$ is a coboundary,
i.e.\ it is of the form $c_g(\gamma) = \frac{a_g}{\gamma(a_g)}$ with
some $a_g \in L^\times$.
Then for any $g \in G$ and any $\gamma \in \Gamma_{[\rho]}$ we obtain
$$ {}^\gamma \rho(g) \otimes \frac{\gamma(a_g)}{a_g} = \rho(g)
\textnormal{ and consequently } \gamma(\rho(g) a_g) = \rho(g) a_g,$$
showing that $\rho(g) a_g$ belongs to $\GL_n(K_{[\rho]})$ (or to $\GSp_n(K_{[\rho]})$).
Taking this equation modulo scalars, yields the theorem.
\end{proof}

\begin{rem}
The previous theorem shows that the field $K_{[  \rho  ]}$ is related to the minimal field of definition of the adjoint representation introduced by R. Pink in a more general context in \cite{Pink98}. 
\end{rem}

The following proposition complements Theorem~\ref{thm:K-rho} by clarifying the relation
between $K_{[\rho]}$ and the minimal field over which a member of the equivalence class~$[\rho]$
can be defined.

\begin{prop}\label{prop:E-rho}
Let $[\rho] \in \GLReps_n(G, L/K)$ or $[\rho] \in \GSpReps_n(G,L/K)$ be (residually) absolutely irreducible. Let $\chi:G\rightarrow K^{\times}$ be any character and let $\psi:G\rightarrow L^{\times}$ be a character of finite order.
If $[\rho]$ is symplectic, assume that the multiplier map is $m(\rho)=\psi\chi$;
otherwise assume that $\det\rho=\psi\chi$ and $E_{[\rho]}$ contains the $n$-th roots of
the values of~$\psi$.

\begin{enumerate}[(a)]
\item\label{prop:E-rho:a} $\Delta_{[\rho]}$ is an open normal subgroup of $\Gamma_{[\rho]}$ and, hence, $E_{[\rho]}/K_{[\rho]}$ is a
finite Galois extension with Galois group $\Gamma_{[\rho]}/\Delta_{[\rho]}$.
In particular, one has $\gamma(E_{[\rho]}) = E_{[\rho]}$ for all $\gamma \in \Gamma_{[\rho]}$.

\item\label{prop:E-rho:b} If $[\rho]$ is symplectic, also assume that $[\rho]$ has no complex multiplication.
The equivalence class $[\rho]$ contains a representation that can be defined over the field $E_{[\rho]}$
and $E_{[\rho]}$ is the smallest such subfield of~$L$.
Moreover, $E_{[\rho]}$ is generated over~$K$ by the traces $\Tr(\rho(g))$ for $g \in G$
(and the values of the multiplier if $\rho$ is symplectic).
\end{enumerate}
\end{prop}

\begin{proof}
(\ref{prop:E-rho:a}) Note that, if $[\rho]$ is symplectic, for all $\delta\in \Delta_{[\rho]}$,  $\delta(m(\rho))=m(\rho)$, whence $E_{[\rho]}$ contains the values of~$\psi$.
Let $\delta \in \Delta_{[\rho]}$ and $\gamma \in \Gamma_{[\rho]}$. There is $\epsilon \in \calE$
such that $(\gamma,\epsilon) \in \calG_{[\rho]}$.  Taking into account that $E_{[\rho]}$ contains the values
of~$\epsilon$ by Lemma~\ref{lem:order}, we compute
\begin{equation*} (\gamma,\epsilon)(\delta,1)(\gamma,\epsilon)^{-1} = (\gamma\delta\gamma^{-1}, ({}^{\gamma\delta\gamma^{-1}}\epsilon^{-1})\epsilon)=
(\gamma\delta\gamma^{-1},1) \in \calG_{[\rho]},\end{equation*}
whence $\gamma \delta \gamma^{-1} \in \Delta_{[\rho]}$. 
The finiteness follows from the fact that $\Delta_{[\rho]}$ has finite index in~$\Gamma$, which was
proved in Lemma~\ref{lem:topology}.

(\ref{prop:E-rho:b}) Lemma~\ref{lem:charpoly-coeff} implies that the character of $\rho$ takes values in~$E_{[\rho]}$,
i.e.\ the traces lie in~$E_{[\rho]}$. If $\rho$ is symplectic, then the multiplier of~$\rho$ also
takes values in~$E_{[\rho]}$.
Due to the (residual) absolute irreducibility, $[\rho]$ can thus be defined over~$E_{[\rho]}$ by Proposition~\ref{prop:rep-theory} (\ref{prop:rep-theory:d}) and~(\ref{prop:rep-theory:e}).
If the traces of $\rho$ (together with the values of the multiplier for symplectic~$\rho$)
generated a proper subfield $\tilde{K}$ of $E_{[\rho]}$, then $[\rho]$ could be defined over~$\tilde{K}$.
Then, however, there would be a $\gamma \in \Gamma \setminus \Delta_{[\rho]}$ such that
$\gamma$ is the identity on~$\tilde{K}$, and $\gamma$ would fix all traces. This would mean ${}^\gamma \rho \sim \rho$ (by Proposition~\ref{prop:rep-theory} (\ref{prop:rep-theory:b}) and~(\ref{prop:rep-theory:c})),
so that $(\gamma,1) \in \calG_{[\rho]}$, consequently, $\gamma \in \Delta_{[\rho]}$, a contradiction.
\end{proof}

\section{Inner twists of a Galois representation}\label{sec:inner-galois}

In this short section we analyse the ramification of the characters occuring in inner twists.
Let $F$ be a number field and let $G := G_F := \Gal(\overline{F}/F)$ in Set-up~\ref{setup}, which
we assume for this section.

\begin{lem}\label{lem:unram}
Let $[\rho] \in \GLReps_n(G_F,L/K)$ or $[\rho] \in \GSpReps_n(G_F,L/K)$.
Then any $\epsilon$ occuring in $(\gamma,\epsilon) \in \calG_{[\rho]}$
is unramified at all places of~$F$ at which $\rho$ is unramified.
\end{lem}

\begin{proof}
Let $\sigma$ be an element such that $\rho(\sigma)$ is the identity. Then ${}^\gamma\rho(\sigma)$
is also the identity, hence, ${}^\gamma\rho \cong \rho \otimes \epsilon$ implies that
$\epsilon(\sigma)=1$. Now apply this reasoning to $\sigma$ in inertia groups at primes at which
$\rho$ does not ramify.
\end{proof}

\begin{ass}\label{ass:shape}
Let $K$ be a finite field of characteristic~$\ell$ or $\Fbar_\ell$ and
$\rho: G \to \GL_n(L)$.
For a prime $\fL$ of~$F$ above~$\ell$ we make the assumption:

There exist an integer~$t$ and an integer~$s$ between $1$ and $n$,
and for each $i=1, \dots, s$,
an $r_i$-tuple $S_i = (a_{i,1},\dots,a_{i,r_i})$ of natural numbers $0 \le a_{i,j} \le \ell-1$
with $r_1 + \cdots + r_s=n$ such that,
if we denote by $B_i$ the matrix,
$$B_i\sim \begin{pmatrix} \psi_{r_i}^{b_i} & \ & \ & 0 \\
\ & \psi_{r_i}^{b_i\ell} & \ & \ \\ \ & \ & \ddots & \ \\ 0 & \ & \
& \psi_{r_i}^{b_i\ell^{r_i-1}}\end{pmatrix}$$
with $\psi_{r_i}$ a (fixed choice of) fundamental character of niveau~$r_i$
(so that $\psi_1$ is the mod-$\ell$ cyclotomic character) and
$b_i=a_{i, 1} + a_{i, 2}\ell + \cdots + a_{i, r_i}\ell^{r_i-1}$,
then
$$(\rho \otimes \psi_1^{t})\vert_{I_{\fL}}\sim\left(\begin{array}{c  c  c }
B_1 & \multicolumn{1}{|c}{} & * \\
\cline{1-1}
\  & \ddots & \ \\
\cline{3-3}
0 & \ & \multicolumn{1}{|c}{B_s} \\
\end{array}\right).$$
\end{ass}

It is well known that Assumption~\ref{ass:shape} is verified if $F_\fL = \QQ_\ell$.

\begin{prop}\label{prop:epsilon-residual}
Let $K$ be a finite field of characteristic~$\ell$ or $\Fbar_\ell$
and $[\rho] \in \GLReps_n(G, L/K)$ or $[\rho] \in \GSpReps_n(G,L/K)$.
Let $\fL$ be a prime of~$F$ above~$\ell$ such that Assumption~\ref{ass:shape} is verified.
If $a_{i,j} < \frac{\ell-1}{2n}$ for all $1 \le i \le s$ and all $1 \le j \le r_i$,
then the character $\epsilon$ is unramified at~$\fL$ for all $(\gamma,\epsilon) \in \calG_{[\rho]}$.
\end{prop}

\begin{proof}
First note that the restriction to $I_\fL$ of the determinant of~$\rho$
is $\chi_\ell^{b_1+\dots+b_s-t}$.
For any exponent~$x$ of $\psi_{r_i}$ occuring in~$B_i$ for any $i \in \{1,\dots,s\}$
we have the 
$$ 0 \le x < \frac{\ell^{r_i}-1}{2n},$$
as each exponent $x$ is of the form $\sum_{j=1}^{r_i} a^{(j)}
\ell^{j-1}$, where the $a^{(j)}$ are the entries of $S_i$ up to a cyclic permutation and
they are less than $\frac{\ell-1}{2n}$.

Let $(\gamma,\epsilon) \in \calG_{[\rho]}$.
We first consider the action of~$\gamma$. As $K$ is a finite field, $\gamma$ acts by raising
to the $\ell^c$-th power for some~$c$. In particular ${}^\gamma \psi_{r_i} = \psi_{r_i}^{\ell^c}$.
This shows that ${}^\gamma (\rho \otimes \psi_1^t) |_{I_\fL}$ has the same shape as
$(\rho \otimes \psi_1^t) |_{I_\fL}$ except that the elements of each $S_i$ are permuted.
Taking the determinant on both sides of ${}^\gamma \rho \cong \rho \otimes \epsilon$
yields that $\epsilon|_{I_\fL}$ has order~$m$ dividing~$n$,
as $\gamma$ acts trivially on the cyclotomic character.
Moreover, looking at any diagonal entry we get
$\psi_{r_i}^x = \psi_{r_j}^y \cdot \epsilon|_{I_\fL}$
for some exponents $x$ and $y$.
Let $r$ be the least common multiple of $r_i$ and $r_j$. Then we can rewrite the
previous equation as
\begin{equation}\label{eq:psi}
\psi_r^{\tilde{x}} = \psi_r^{\tilde{y}} \cdot \epsilon|_{I_\fL}
\end{equation}
with $0 \le \tilde{x} < \frac{\ell^r-1}{2n}$, since
$\psi_r=\psi_{r_i}^{\frac{\ell^r-1}{\ell^{r_i}-1}}$ and hence
$\tilde{x}=x\frac{\ell^r-1}{\ell^{r_i}-1}<\frac{\ell^{r_i}-1}{2n}\frac{\ell^r-1}{\ell^{r_i}-1}=
\frac{\ell^{r}-1}{2n}$.
The same estimate holds true for~$\tilde{y}$.

Furthermore, the knowledge of the order of $\epsilon|_{I_\fL}$
implies $\epsilon|_{I_\fL} = \psi_r^{\frac{\ell^r-1}{m}b}$
for some $1 \le b\le m-1$ coprime to~$m$.
Note that this implies the estimates
$$ \frac{\ell^r-1}{n} \le \frac{\ell^r-1}{m} \le \frac{\ell^r-1}{m}b \le
\frac{(m-1)(\ell^r-1)}{m} \le \frac{(n-1)(\ell^r-1)}{n}.$$
But, since $\psi_r$ has order $\ell^r-1$, Equation \eqref{eq:psi} implies that $\frac{(\ell^r-1)b}{m} +\tilde{y}-\tilde{x}$ is divisible by $\ell^r-1$, hence we get a contradiction. 
Thus $\epsilon|_{I_\fL}$ is trivial, as was to be shown.
\end{proof}

\section{Inner twists in compatible systems of Galois representations}\label{sec:compatible}

In this section we aim to generalise the results of Section~\ref{sec:inner-group} to
compatible systems of Galois representations. For a compatible system $\rho_\bullet$
we will define number fields $E_{\rho_\bullet}$ and $K_{\rho_\bullet}$ such that almost
all members $\rho_\lambda$ and their projectivisation $\rho_\lambda^\proj$
can be defined over the completion at $\lambda$ of $E_{\rho_\bullet}$ and $K_{\rho_\bullet}$,
respectively.
The example of a non-CM newform $f$ already used in Section~\ref{sec:inner-group} can serve as
a guidance to the reader. The field $E_{\rho_\bullet}$ is then
$\QQ_f = \QQ(a_p \;|\; p \textnormal{ prime }, p \nmid N)$
and $K_{\rho_\bullet}$ is $\QQ(\frac{a_p^2}{\psi(p)} \;|\; p \textnormal{ prime }, p \nmid N)$,
where $N$ is the level and $\psi$ the Dirichlet character of~$f$.

\begin{defi}\label{defi:compatible}
Let $n \in \NN$ and let $G := G_F:= \Gal(\overline{F}/F)$ be the absolute Galois group of a number field~$F$.
A {\em compatible system $\rho_\bullet = (\rho_\lambda)_\lambda$
of $n$-dimensional representations of $G_F$} consists of the following data:
\begin{itemize}
\itemsep=0cm plus 0pt minus 0pt
\parsep=0.1cm
\item A number field $L$.
\item A finite set $S$ of finite places of~$F$.
\item For each finite place $\fp$ of~$F$ not in~$S$, a monic polynomial $P_\fp \in \calO_L[X]$
(with $\calO_L$ the ring of integers of~$L$).
\item For each finite place $\lambda$ of~$L$
(together with fixed embeddings $L \hookrightarrow L_\lambda \hookrightarrow \overline{L}_\lambda$
with $L_\lambda$ the completion of $L$ at~$\lambda$ and $\overline{L}_\lambda$ an
algebraic closure thereof) a continuous Galois representation
$$ \rho_\lambda: G_F \to \GL_n(\overline{L}_\lambda)$$
such that $\rho_\lambda$ is unramified outside $S \cup S_\ell$
(where $\ell$ is the residue characteristic of~$\lambda$ and $S_\ell$
is the set places of~$F$ lying above~$\ell$)
and such that for all~$\fp \not\in S \cup S_\ell$ the characteristic
polynomial of $\rho_\lambda(\Frob_\fp)$ is equal to $P_\fp$ (equality inside $\overline{L}_\lambda[X]$).
\end{itemize}

We say that the compatible system $\rho_\bullet = (\rho_\lambda)_\lambda$ is {\em symplectic}
if for all $\lambda$ the representation $\rho_\lambda$ is of the form $G_F \to \GSp_n(\overline{L}_\lambda)$.
We say that $\rho_\bullet = (\rho_\lambda)_\lambda$ is {\em almost everywhere (a.~e.) symplectic}
if for all but possibly finitely many $\lambda$ the representation $\rho_\lambda$
is of the form $G_F \to \GSp_n(\overline{L}_\lambda)$.

Let $a \in \ZZ$ and $\psi: G_F \to L^\times$ a continuous finite order character.
We say that the compatible system $\rho_\bullet=(\rho_\lambda)_\lambda$ has {\em determinant}
$\psi \chi_\ell^a$ if for all places $\lambda$ of~$L$ (say, $\lambda$ lies over the
rational prime~$\ell$) the determinant of $\rho_\lambda$
is $\psi \chi_\ell^a$ with $\chi_\ell$ the $\ell$-adic cyclotomic character.

We say that a symplectic compatible system $\rho_\bullet=(\rho_\lambda)_\lambda$ has {\em multiplier map}
$\psi \chi_\ell^a$ if for all places $\lambda$ of~$L$ (say, $\lambda$ lies over the
rational prime~$\ell$) the multiplier map of $\rho_\lambda$
is $\psi \chi_\ell^a$ with $\chi_\ell$ the $\ell$-adic cyclotomic character.
In the case of an almost everywhere symplectic compatible system, we only insist on this property
at those places~$\lambda$, where the representation is symplectic.

A compatible system $\rho_\bullet = (\rho_\lambda)_\lambda$ is called {\em almost everywhere (a.~e.) absolutely irreducible}
if all its members $\rho_\lambda$ are absolutely irreducible except for finitely many places $\lambda$ of $L$.

For a place $\fp$ of~$F$ not in~$S$, denote by $a_\fp$ the coefficient in front of $X^{n-1}$
of $P_\fp(X)$, i.e.\ the  `trace' (up to sign).
For the sequel, assume Set-up~\ref{setup} for some number field $K$, and assume that $L/K$ is a finite Galois extension of number fields with Galois group $\Gamma$. Note that, in this context, $\Gamma$, $L$ and $K$ are all endowed with the discrete topology. 
Let
$$\calG_{\rho_\bullet} := \{(\gamma,\epsilon) \in \calG \;|\; \gamma (a_\fp) = a_\fp \cdot \epsilon(\Frob_\fp)
\textnormal{ for all places $\fp$ of~$F$ not in~$S$} \},$$
$\Gamma_{\rho_\bullet} := \pi(\calG_{\rho_\bullet}) \subseteq \Gamma$,
$\Delta_{\rho_\bullet} := \{ \gamma \in \Gamma_{\rho_\bullet} \;|\; (\gamma,1) \in \calG_{\rho_\bullet} \}$ and
$\calE_{\rho_\bullet} := \iota^{-1}(\calG_{\rho_\bullet}) = \iota^{-1}(\ker(\pi|_{\calG_{\rho_\bullet}}))$.
(Compare with Definition~\ref{defi:main} and Lemma~\ref{lem:G-on-traces}).
We say that the compatible system $\rho_\bullet$ has {\em no complex multiplication} if
$\calE_{\rho_\bullet} = \{1\}$.

Let $K_{\rho_\bullet} := L^{\Gamma_{\rho_\bullet}}$, called the {\em projective field of definition of~$\rho_\bullet$},
and $E_{\rho_\bullet} := L^{\Delta_{\rho_\bullet}}$, called the {\em field of definition of~$\rho_\bullet$}. Note that $E_{\rho_{\bullet}}$ is generated over~$K$
by the set $\{a_{\mathfrak{p}} \;|\; \mathfrak{p} \text{ place of }F\text{ not in }S\}$.
\end{defi}

\begin{rem}\label{rem:noCM} Let  $\rho_{\bullet}$ be an a.~e.\ symplectic compatible system of $n$-dimensional Galois representations without complex multiplication, and let $\rho_{\lambda}\in \rho_{\bullet}$ be a symplectic member of the system. Then $\rho_{\lambda}$ has no complex multiplication. Indeed, if $\epsilon:G_F\rightarrow \overline{L}_{\lambda}^{\times}$ is an inner twist for $[\rho_{\lambda}]\in \GSpReps_n(G_F, \overline{L}_{\lambda}/K_{\lambda})$, then $m(\rho_{\lambda})=m(\rho_{\lambda}\otimes \epsilon)=m(\rho_{\lambda})\epsilon^2$, hence $\epsilon:G_F\rightarrow \overline{L}_{\lambda}^{\times}$ is at most quadratic. Therefore we may view it as a character $\epsilon:G_F\rightarrow L^{\times}$, and it satisfies $a_{\mathfrak{p}} = a_{\mathfrak{p}} \cdot \epsilon(\Frob_{\mathfrak{p}})$ (equality in $L$) for all $\mathfrak{p}\not\in S$. Thus $\epsilon\in \mathcal{E}_{\rho_{\bullet}}=1$ because $\rho_{\bullet}$ has no complex multiplication. This implies that $\rho_{\lambda}$ does not have complex 
multiplication. 
\end{rem}

Let $\rho_\bullet = (\rho_\lambda)_\lambda$ be an a.~e.\ absolutely irreducible compatible system.
For a place $\lambda$ of~$L$ we can consider Set-up~\ref{setup} for the Galois extension
$\overline{L}_\lambda/K_\lambda$.  
If $\rho_{\lambda}\in \rho_{\bullet}$ is residually absolutely irreducible (and symplectic), by Proposition \ref{prop:rep-theory}(\ref{prop:rep-theory:d}) (Proposition \ref{prop:rep-theory}(\ref{prop:rep-theory:e}) and Remark \ref{rem:noCM}) the equivalence class of $\rho_{\lambda}$ contains a member that is defined over $L_{\lambda}$. Thus we can consider Set-up~\ref{setup} for the Galois extension
$L_\lambda/K_\lambda$. In this case we will 
use the notation $\Gamma_\lambda := \Gal(L_\lambda/K_\lambda)$
(denoting the prime of~$K$ below~$\lambda$ also by~$\lambda$),
and $\calE_\lambda := \{\epsilon: G_F \to L^\times_\lambda \textnormal{ character }\}$
and $\calG_\lambda := \calE_\lambda \rtimes \Gamma_\lambda$.
In this setting, apply Definition~\ref{defi:main} with $[\rho_\lambda]$, so that, in particular,
$\calG_{[\rho_\lambda]}$ is the stabiliser subgroup inside~$\calG_\lambda$ of~$[\rho_\lambda]$.

We start with the analog of Proposition~\ref{prop:E-rho} for compatible systems
of $\ell$-adic Galois representations; our main focus below will be to obtain results
about their residual representations.

\begin{prop}\label{prop:fieldE}
Let $\rho_\bullet$ be an a.~e.\ absolutely irreducible compatible system of $n$-dimensional Galois representations as in Definition~\ref{defi:compatible}.
If $\rho_{\bullet}$ is a.~e.\ symplectic, assume that the multiplier of $\rho_{\bullet}$ is $\psi\chi_{\ell}^a$ and that $\rho_\bullet$ has no complex multiplication.
If $\rho_{\bullet}$ is not a.~e.\ symplectic, assume that the determinant of $\rho_{\bullet}$ is $\psi\chi_{\ell}^a$ and that $E_{\rho_{\bullet}}$ contains the $n$-th roots of the values of~$\psi$.

\begin{enumerate}[(a)]
\item\label{prop:fieldE:a} $\Delta_{\rho_{\bullet}}$ is a normal subgroup of $\Gamma_{\rho_{\bullet}}$ and, hence, $E_{\rho_{\bullet}}/K_{\rho_{\bullet}}$ is a
Galois extension with Galois group $\Gamma_{\rho_{\bullet}}/\Delta_{\rho_{\bullet}}$.
In particular, the field $\gamma(E_{\rho_{\bullet}}) = E_{\rho_{\bullet}}$ for all $\gamma \in \Gamma_{\rho_{\bullet}}$.

\item\label{prop:fieldE:b} For each place $\lambda$ of $L$ such that $\rho_{\lambda}$ is residually absolutely irreducible
(and symplectic if $\rho_\bullet$ is a.~e.\ symplectic),
the equivalence class $[\rho_{\lambda}]$ contains a representation that can be defined over the field $(E_{\rho_{\bullet}})_\lambda$
and $(E_{\rho_{\bullet}})_\lambda$ is the smallest such field. 
\end{enumerate}
\end{prop}

\begin{proof} (\ref{prop:fieldE:a})  
First of all, note that, if $\rho_{\bullet}$ is a.~e.\ symplectic, then
$E_{\rho_{\bullet}}$ contains the values of~$\psi$. Indeed, for all $\delta\in \Delta_{\rho_{\bullet}}$,
for all places $\lambda$ of $L$ fixed by~$\delta$ 
and all places $\fp$ of $F$ outside $S\cup S_{\ell}$
(where $\ell$ is the residue characteristic of $\lambda$), we have that $\Tr(\rho_{\lambda}(\Frob_{\mathfrak{p}}))=a_{\mathfrak{p}}
=\delta(a_{\mathfrak{p}})=\Tr({}^{\delta}\rho_{\lambda}(\Frob_{\mathfrak{p}}))$.
Here by abuse of notation we denote by $\delta$ any extension of $\delta$ to~$\overline{L}_\lambda$.
Suppose that $\rho_{\lambda}$ is absolutely irreducible and symplectic. Then by Remark \ref{rem:noCM} $\rho_{\lambda}$ does not have complex multiplication.
Moreover, by Chebotarev's Density Theorem and Proposition~\ref{prop:rep-theory}~(\ref{prop:rep-theory:c}),
$\rho_{\lambda}$ and ${}^{\delta}\rho_{\lambda}$ are equivalent as symplectic representations.
Hence $\delta(m(\rho_\lambda)) = m({}^{\delta}\rho_{\lambda})=m(\rho_{\lambda})$ for all except
possibly finitely many places $\lambda\in L$.
Since the multiplier of $\rho_{\bullet}$ is $\psi\chi_{\ell}^a$,
we obtain that the values of~$\psi$ are fixed by $\delta$, thus they lie in~$E_{\rho_{\bullet}}$.

This allows us to compute $(\gamma \epsilon)(\delta, 1)(\gamma, \epsilon)^{-1}=(\gamma\delta\gamma^{-1}, 1)\in \mathcal{G}_{\bullet}$ as in  Proposition~\ref{prop:E-rho}~(\ref{prop:E-rho:a}) for all $\gamma\in \Gamma_{\rho_{\bullet}}$ and $\delta\in \Delta_{\rho_{\bullet}}$, which shows that $\gamma\delta\gamma^{-1}\in \mathcal{E}_{\rho_{\bullet}}$.

(\ref{prop:fieldE:b}) Let $\lambda$ be a place of $L$ such that $\rho_{\lambda}$ is residually absolutely irreducible (and symplectic if $\rho_\bullet$ is a.~e.\ symplectic). For all $\fp\not\in S\cup S_{\ell}$ (where $\ell$ is the residue characteristic of $\lambda$), the trace of $\rho_{\lambda}(\Frob_{\fp})$ is fixed by all elements of $\Delta_{\rho_{\bullet}}$ that fix~$\lambda$, whence it belongs to $(E_{\rho_{\bullet}})_\lambda$, and we saw above that the multiplier of $\rho_{\lambda}$ takes values in $(E_{\rho_{\bullet}})_\lambda$ when $\rho_{\bullet}$ is a.~e.\ symplectic. Using Chebotarev's density theorem and Proposition~\ref{prop:rep-theory} (\ref{prop:rep-theory:d}) and~(\ref{prop:rep-theory:e}),
we see that $\rho_{\lambda}$ can be defined over $(E_{\rho_{\bullet}})_\lambda$.

Since $E_{\rho_{\bullet}}$ is generated by the $\{a_{\mathfrak{p}}\;|\;\mathfrak{p}\not\in S\}$, it is clear that $[\rho_{\lambda}]$ cannot contain a member which is defined over a subfield smaller than $(E_{\rho_{\bullet}})_{\lambda}$.
\end{proof}

We will next see that, when $\rho_{\lambda}$ is residually absolutely irreducible, the projective field of definition of $[\rho_\lambda]$ of Definition~\ref{defi:main}
is just the completion at~$\lambda$
of the projective field of definition of the compatible system.

\begin{lem}\label{lem:compatible}
Let $\rho_\bullet = (\rho_\lambda)_\lambda$ be an $n$-dimensional a.~e.\ absolutely irreducible compatible system as in Definition~\ref{defi:compatible}.
Let $\psi:G_F \to L^\times$ be a character of finite order.
If $\rho_\bullet$ is a.~e.\ symplectic, assume that $\rho_\bullet$ has no complex multiplication and
that the multiplier of $\rho_\bullet$ is $\psi \chi_\ell^a$.
If $\rho_\bullet$ is not a.~e.\ symplectic, assume that $L$ contains
the $n$-th roots of the values of~$\psi$ and
that the determinant of $\rho_\bullet$ is $\psi \chi_\ell^a$.
Let $\lambda$ be a place of $L$ such that $\rho_{\lambda}$ is residually absolutely irreducible
(and symplectic if $\rho_\bullet$ is a.~e.\ symplectic).

Then we have
$$\Gamma_{[\rho_\lambda]} = \Gamma_\lambda \cap \Gamma_{\rho_\bullet},$$
where the intersection is taken inside $\Gamma$ via the embedding
$\Gamma_\lambda \hookrightarrow \Gamma$ coming from the embedding $L \hookrightarrow L_\lambda$.
\end{lem}

\begin{proof}
We use Lemma~\ref{lem:G-on-traces}. The inclusion `$\supseteq$' is
clear. Conversely, let $\gamma \in \Gamma_{[\rho_\lambda]}$; in
particular, $\gamma \in \Gamma_\lambda$. There is a character
$\epsilon \in \calE_\lambda$ such that $(\gamma,\epsilon) \in
\calG_{[\rho_\lambda]}$, i.e.\ that $\gamma (a_\fp) = a_\fp \cdot
\epsilon(\Frob_\fp)$ for all $\fp \not\in S \cup S_\ell$.
Lemma~\ref{lem:order} implies that $\epsilon$ can be lifted to an element of~$\calE$,
i.e.\ taking values
in~$L^\times$. The equation $\gamma (a_\fp) = a_\fp \cdot
\epsilon(\Frob_\fp)$ for all $\fp \not\in S \cup S_\ell$ is now an
equation in the field~$L$, implying $(\gamma,\epsilon) \in
\calG_{\rho_\bullet}$.
\end{proof}

The following theorem is our main result about projective fields of definition
of compatible systems of $\ell$-adic representations.

\begin{thm}\label{thm:compatible}
Let $\rho_\bullet = (\rho_\lambda)_\lambda$ be an $n$-dimensional a.~e.\ absolutely irreducible compatible system as in Definition~\ref{defi:compatible}.
Let $\psi:G_F \to L^\times$ be a character of finite order.
If $\rho_\bullet$ is a.~e.\ symplectic, assume that $\rho_\bullet$ has no complex multiplication and
that the multiplier of $\rho_\bullet$ is $\psi \chi_\ell^a$.
If $\rho_\bullet$ is not a.~e.\ symplectic, assume that $L$ contains
the $n$-th roots of the values of~$\psi$ and
that the determinant of $\rho_\bullet$ is $\psi \chi_\ell^a$.

Then for each place $\lambda$ of~$L$ such that $\rho_{\lambda}$ is residually absolutely irreducible (and symplectic if $\rho_\bullet$ is a.~e.\ symplectic), the projective field of definition of $[\rho_\lambda]$
is~$(K_{\rho_\bullet})_\lambda$, the completion of $K_{\rho_\bullet}$ at the prime lying below~$\lambda$.
\end{thm}

\begin{proof}
By Lemma~\ref{lem:compatible}, we have $K_{[\rho_\lambda]} = L_\lambda^{\Gamma_{[\rho_\lambda]}} =
L_\lambda^{\Gamma_\lambda \cap \Gamma_{\rho_\bullet}} = (L^{\Gamma_{\rho_\bullet}})_\lambda
= (K_{\rho_\bullet})_\lambda$.
\end{proof}

Given a compatible system $\rho_\bullet=(\rho_\lambda)_\lambda$ we will now consider
the residual representations $\overline{\rho}_\lambda$ (see Remark~\ref{rem:rep_integrality}).
If $M$ is a local field, we denote by $\kappa(M)$ its residue field.
Let $\lambda$ be a finite place of~$L$ and assume $\rho_{\lambda}$ is defined over $L_{\lambda}$ (which is the case whenever $\rho_{\lambda}$ is residually absolutely irreducible (and symplectic without complex multiplication)).
Consider Set-up~\ref{setup} for the Galois extension $\kappa(L_\lambda)/\kappa(K_\lambda)$
with Galois group~$\overline{\Gamma}_\lambda$, using the notation $\overline{\calE}_\lambda$
and $\overline{\calG}_\lambda$.
We apply Definition~\ref{defi:main} to the equivalence class $[\overline{\rho}_\lambda]$
of the residual representation,
using the notation $\calG_{[\overline{\rho}_\lambda]} \subseteq \overline{\calG}_\lambda$,
$\calE_{[\overline{\rho}_\lambda]}\subseteq \overline{\calE}_\lambda$ and
$\Gamma_{[\overline{\rho}_\lambda]}\subseteq \overline{\Gamma}_\lambda$.

We now state and prove our principal result describing the projective fields of definition
for almost all residual representations of a compatible system.
It may be useful to roughly summarise where the various conditions in the theorem come from.
That $\psi$ has finite order is needed to ensure that all $\epsilon$ occuring in inner
twists are of finite order; the condition on the $n$-th roots of the values of~$\psi$
in the non-symplectic case ensures that $\epsilon$ takes its values in $E_{\rho_\bullet}$
(which is automatic in the symplectic case). The absolute irreducibility and the
non-CM condition (in the symplectic case) are needed to ensure that the representations
are determined by the characteristic polynomials of Frobenius.
The condition on the shape above~$\ell$ is needed to exclude that the residual
inner twists ramify at~$\ell$ (if infinitely many of them did ramify, then they would not
`glue' to an inner twist of the compatible system).

\begin{thm}\label{thm:residual}
Let $\rho_\bullet = (\rho_\lambda)_\lambda$ be an $n$-dimensional a.~e.\ absolutely irreducible
compatible system as in Definition~\ref{defi:compatible}.
Let $\psi:G_F \to L^\times$ be a character of finite order.
If $\rho_\bullet$ is a.~e.\ symplectic, assume that $\rho_\bullet$ has no complex multiplication and
that the multiplier of $\rho_\bullet$ is $\psi \chi_\ell^a$.
If $\rho_\bullet$ is not a.~e.\ symplectic, assume that $L$ contains the $n$-th roots of the
values of~$\psi$ and that the determinant of $\rho_\bullet$ is $\psi \chi_\ell^a$.

Assume that for primes $\fL$ of~$F$ lying over the residue characteristic of~$\lambda$,
the representation $\overline{\rho}_\lambda$ satisfies Assumption~\ref{ass:shape}
(this is automatic if $F_\fL = \QQ_\ell$).
Moreover, assume that there is an integer~$k$, independent of~$\lambda$,
such that the numbers $a_{i,j}$ appearing in Assumption~\ref{ass:shape}
are bounded by~$k$.

Then for all places~$\lambda$ of~$L$ at which $\rho_\lambda$ is residually absolutely irreducible, except possibly finitely many, the projective
field of definition of $[\overline{\rho}_\lambda]$ is $\kappa((K_{\rho_\bullet})_\lambda)$.
\end{thm}

\begin{proof}
We may restrict to those $\lambda$ of residual characteristic~$\ell$
satisfying $\frac{\ell-1}{2n} > k$
(cf.\ Proposition~\ref{prop:epsilon-residual}), and symplectic without complex multiplication if $\rho_\bullet$ is a.~e.\ symplectic.
We may furthermore limit ourselves to $\lambda$ such that $\lambda$ is unramified in~$L/K$.

Let $(\gamma,\epsilon) \in \calG_{[\overline{\rho}_\lambda]}$ be an inner
twist of~$[\overline{\rho}_\lambda]$. We identify $\epsilon$
with its lift to an element of $\calE$ of the same order (with respect to the
fixed $L \hookrightarrow L_\lambda$) and $\Gamma_{\lambda}=\Gal(L_{\lambda}/K_{\lambda})$ with $\overline{\Gamma}_{\lambda}=\Gal(\kappa(L_{\lambda})/\kappa(K_{\lambda}))$, so that $\mathcal{G}_{[\rho_{\lambda}]}\subseteq \mathcal{G}_{[\overline{\rho}_{\lambda}]}$.
By Lemma~\ref{lem:order}
we know that the order of~$\epsilon$ is bounded independently of~$\lambda$.
Moreover, by Proposition~\ref{prop:epsilon-residual}, $\epsilon$
is unramified outside~$S$.
Let $\calE_0$ be the finite subset of~$\calE$ consisting of those characters
meeting these two requirements.
Hence, $\calG_0 := \calE_0 \rtimes \Gamma$ is a finite subgroup of~$\calG$
and  $\calG_{[\rho_{\lambda}]}$ and $\calG_{[\overline{\rho}_\lambda]}$ are subgroups.

Now consider $(\gamma,\epsilon) \in \calG_0$.
Assume that $(\gamma,\epsilon) \in \calG_{[\overline{\rho}_\lambda]}$
for infinitely many~$\lambda$. Then for all places $\fp$ of~$F$ not in~$S$
and each of these $\lambda$ (except those above the residue characteristic of~$\fp$),
we have
$$ \gamma (a_\fp) \equiv a_\fp \cdot \epsilon(\Frob_\fp) \pmod \lambda.$$
Consequently, we have equality $\gamma(a_\fp) = a_\fp \cdot \epsilon(\Frob_\fp)$,
whence $(\gamma,\epsilon) \in \calG_{[\rho_{\lambda}]}$.
This means that avoiding also those finitely many~$\lambda$ such that
a $(\gamma,\epsilon) \in \calG_0 \setminus \calG_{[\rho_{\lambda}]}$ is in
$\calG_{[\overline{\rho}_\lambda]}$, we have
$\calG_{[\rho_\lambda]} = \calG_{[\overline{\rho}_\lambda]}$, thus 
$\kappa(K_{[\rho_{\lambda}]})=K_{[\overline{\rho}_{\lambda}]}$.
Since $\rho_{\lambda}$ is residually absolutely irreducible, Theorem~\ref{thm:compatible} yields $K_{[\rho_{\lambda}]}=(K_{\rho_{\bullet}})_{\lambda}$, which proves the theorem.
\end{proof}

\section{Application to compatible systems with huge residual images}\label{sec:huge}

In this section we make use of Theorem~\ref{thm:residual} to prove a
first result that allows us to (almost) determine the projective
image of the residual representation $\overline{\rho}_{\lambda}$,
(except for finitely many $\lambda$), for a compatible system
$\rho_\bullet=(\rho_{\lambda})$ satisfying suitable conditions.

Let $V$ be an $n$-dimensional $K$-vector space endowed with a
symplectic pairing. Recall that an element $\tau \in \GL(V)$ is a nontrivial
{\em transvection} if $\tau - \id_V$ has rank~$1$, i.e.\ if $\tau$
fixes a hyperplane pointwisely and there is a line $U$ such that
$\tau(v)-v \in U$ for all~$v \in V$. Any transvection has
determinant~$1$. A {\em symplectic transvection} is a transvection
in~$\Sp(V)$. Any symplectic transvection has the form
$$ T_v[\lambda] \in \Sp(V): u \mapsto u + \lambda \langle u,v \rangle v$$
with {\em direction vector} $v \in V$ and {\em parameter} $\lambda
\in K$ (cf.~\cite{A}, p.~137-138).

\begin{defi}\label{defi:huge}
Let $L$ be an algebraically closed field, and $G$ a subgroup of
$\GSp_{n}(L)$. We will say that $G$ is \emph{huge} if the subgroup
of $G$ generated by the transvections contained in $G$ is conjugated
(in $\GSp_n(L)$) to $\Sp_n(K)$ for some subfield $K\subseteq L$.
\end{defi}

Let $\overline{K}$ be an algebraic closure of~$K$.

\begin{rem} When $K$ is a finite field of characteristic $\ell$, a subgroup
$G\subseteq \GSp_n(\overline{K})$ is huge if and only if it contains a subgroup
conjugated (in $\GSp_n(\overline{K})$) to $\Sp_n(\mathbb{F}_{\ell})$.
This result is proved in Part~II~\cite{partII} of the present series of papers.
\end{rem}

The main ingredient will be the following group-theoretic result:

\begin{prop}\label{prop:moregt}
Let $K$ be a finite field of characteristic different from $2$,
and $G\subseteq \GSp_{n}(\overline{K})$ be a group
such that the group generated by the transvections in $G$ is
$\Sp_{n}(K)$. Then $G\subseteq \GSp_{n}(K){\overline{K}}^{\times}$.
\end{prop}

To prove this proposition, first note the following easy fact:

\begin{lem} Let $G\subseteq \GSp_{n}(\overline{K})$ be a group such that the group generated
by the transvections in $G$ is $\Sp_{n}(K)$. Then $G$ is contained
in the normaliser of $\Sp_n(K)$ in $\GSp_{n}(\overline{K})$.
\end{lem}

\begin{proof}
Let $A\in G$ be any element, say with multiplier $\alpha$. To see
that it belongs to the normaliser
$N_{\GSp_{n}(\overline{K})}(\Sp_{n}(K))$, it suffices to see that,
for all transvection $T=T_v[\lambda]\in \Sp_{n}(K)$, $ATA^{-1}\in
\Sp_{n}(K)$.
An easy computation shows that, for all $w\in V$,
\begin{multline*}AT_v[\lambda]A^{-1}(w)=A(A^{-1}w + \lambda \langle A^{-1}w, v\rangle Av)= w + \lambda\langle A^{-1}w, v\rangle Av=\\
w + \lambda \alpha \langle w, Av\rangle Av=T_{Av}[\lambda
\alpha].
\end{multline*}
But $T_{Av}[\lambda]$ is a transvection and belongs to $G$, so it
belongs to $\Sp_{n}(K)$, as was to be shown.
\end{proof}

So, in order to obtain Proposition~\ref{prop:moregt},
all we need is to prove the following proposition.

\begin{prop}\label{prop:normalizer}
Let $K$ be a finite field of characteristic different from $2$.
Then $N_{\GSp_{n}(\overline{K})}(\Sp_{n}(K))\subseteq \GSp_{n}(K)\overline{K}^{\times}$.
\end{prop}

We will make use of an auxiliary lemma:

\begin{lem}\label{lem:graphautom}
Let $K$ be field of characteristic different from $2$.
The group of automorphisms of $\Sp_{n}(K)$ is generated by two
subgroups, as described below:
\begin{itemize}
 \item The group of automorphisms $\Phi$ such that: for all $B\in \Sp_{n}(K)$,  $\Phi(B)=ABA^{-1}$,
where $A\in \GSp_{n}(K)$ (semi-inner automorphisms).

 \item The group of automorphisms $\Phi$ such that there exists $\phi:K\rightarrow K$ field automorphism such that,
for all $B\in \Sp_{n}(K)$,
  $\Phi(B)$ is the matrix with entries obtained by applying $\phi$ to
the entries of $B$ (field automorphisms).
\end{itemize}
\end{lem}

\begin{proof} This follows from the main Theorem of \cite{Hua1948}, page 740.
\end{proof}

\begin{proof}[Proof of Proposition \ref{prop:normalizer}]
Denote by $N$ the normaliser of $\Sp_{n}(K)$ in $\GSp_{n}(\overline{K})$, and
consider also the centraliser
$C=C_{\GSp_{n}(\overline{K})}\Sp_{n}(K)$, that is, the elements
$A\in \GSp_{n}(\overline{K})$ such that, for all $B\in \Sp_{n}(K)$,
$A^{-1}BA=B$.

First, note that we can view $N/C$ in a natural way as a subgroup of
$\Aut(\Sp_{n}(K))$: namely, given $A\in N$, it defines an
automorphism of $\Sp_{n}(K)$ by conjugation, and the kernel of the
homomorphism $N\rightarrow \Aut(\Sp_{n}(K))$ is obviously $C$. Let
us call $\tilde{N}$ the image of $N$ in the automorphism group.

We know that $\tilde{N}$ is a subgroup of $\Aut(\Sp_{n}(K))$. By
Lemma \ref{lem:graphautom}, we know that the group of automorphisms
is generated by two subgroups, one consisting of the semi-inner
automorphisms and one consisting of the field automorphisms. But no
automorphism of $\tilde{N}$ can be a field automorphism. This is
proved in \cite{KLS1}, page 548, inside of the proof of Corollary
2.6; namely, the reasoning is that any automorphism of the shape
$B\mapsto A^{-1}BA$ must respect the trace. But there is always an
element $B\in \Sp_{2n}(K)$ such that no field automorphism preserves
the trace of $B$.

On the other hand, all semi-inner automorphisms belong to
$\tilde{N}$. From here we can conclude that the semi-inner
automorphisms are precisely those belonging to $\tilde{N}$. Indeed,
note that the composition of a field automorphism with a semi-inner
automorphism coincides with the composition of a semi-inner
automorphism with a field automorphism, so each automorphism can be
written as a product of a semi-inner one and a field one, and then
we would get that if $\tilde{N}$ contains an automorphism which is
not semi-inner, it contains an automorphism which is a field
automorphism. Therefore the group $\tilde{N}$ coincides with the group
of semi-inner automorphisms. That is to say, for any matrix $A\in
N$, there exists a matrix $A_1\in \GSp_{n}(K)$ such that the
automorphism $B\mapsto A^{-1}BA$ coincides with the automorphism
$A_1^{-1}BA_1$. That is to say, for all $B\in \Sp_{n}(K)$,
$(A_1A^{-1})^{-1}B(A_1A^{-1})=B$. Equivalently, $A_1A^{-1}\in C$.

But, in our situation, $C={\overline{K}}^{\times}\{\mathrm{Id}\}$.
One sees this by considering a basis of $V$ consisting of directions
of transvections in $\Sp_{n}(K)$, say $T_{v_1}[\lambda_1], \dots,
T_{v_{n}}[\lambda_{n}]$. Assume that $B\in C$. Then, for all $i=1,
\dots, n$, $B^{-1}T_{v_i}[\lambda]B=T_{v_i}[\lambda]$. But
$B^{-1}T_{v_i}[\lambda]B=T_{Bv_i}[\lambda \beta^{-1}]$ (where
$\beta$ is the multiplier of $B$), hence $B$ must fix all $\langle
v_i\rangle$. Repeating the same reasoning with transvections with
directions $v_1 + v_2, \dots, v_1 + v_n$, one sees that $B$ must
be a homothety.

Hence $N=\GSp_{2n}(K){\overline{K}}^{\times}$.
\end{proof}

As a consequence of Proposition~\ref{prop:moregt}, we obtain the following.

\begin{cor}\label{cor:moregt} Let $K$ be a finite field of characteristic different from $2$, and $G\subseteq \GSp_{n}(\overline{K})$ be a group
such that the group generated by the transvections in $G$ is
$\Sp_{n}(K)$. Then the image of $G$ under the projection
$\GSp_{n}(\overline{K})\rightarrow \PGSp_n(\overline{K})$ is either
$\PSp_n(K)$ or $\PGSp_n(K)$.
\end{cor}

\begin{cor}\label{cor:residual}
Let $\rho_\bullet = (\rho_\lambda)_\lambda$ be an $n$-dimensional a.~e.\ absolutely irreducible
and a.~e.\ symplectic compatible system as in Definition~\ref{defi:compatible}.
Let $\psi:G_F \to L^\times$ be a character of finite order.
Assume that the multiplier of $\rho_\bullet$ is $\psi \chi_\ell^a$.

Assume, moreover, that for all but possibly a density zero set of places~$\lambda$ of~$L$
the residual representation $\overline{\rho}_\lambda$ has huge image.

Assume that for primes $\fL$ of~$F$ lying over the residue characteristic of~$\lambda$,
the representation $\overline{\rho}_\lambda$ satisfies Assumption~\ref{ass:shape}
(this is automatic if $F_\fL = \QQ_\ell$).
Moreover, assume that there is an integer~$k$, independent of~$\lambda$,
such that the numbers $a_{i,j}$ appearing in Assumption~\ref{ass:shape}
are bounded by~$k$.

Then for all places $\lambda$ of~$L$ with the possible exception of a density zero set,
the image of $\overline{\rho}_\lambda^\proj$
is $\PGSp_n(\kappa((K_{\rho_\bullet})_\lambda))$ or $\PSp_n(\kappa((K_{\rho_\bullet})_\lambda))$.
\end{cor}

\begin{proof}
Let $\lambda$ be a place of~$L$ such that the image of $\overline{\rho}_\lambda$ is symplectic
and huge. Then the restriction of $\overline{\rho}_\lambda$
to $I_{[\overline{\rho}_\lambda]}$ (see Proposition~\ref{prop:K-rho}) is absolutely irreducible and
Theorems \ref{thm:K-rho} and~\ref{thm:residual}
show that $\overline{\rho}_\lambda^\proj$ can be defined over
$\kappa((K_{\rho_\bullet})_\lambda)$ and that this field is the smallest one with this property.

Now, because of Corollary~\ref{cor:moregt}, we know that there
exists a field $K'\subseteq \overline{K}$ such that the image of
$\overline{\rho}_{\lambda}^\proj$ is conjugated to $\PSp_{n}(K')$ or
$\PGSp_n(K')$ whenever $\overline{\rho}_\lambda$ has huge image. In
particular $\overline{\rho}_{\lambda}^\proj$ can be defined over
$K'$, and this is the smallest field with this property. Hence
$K'=\kappa((K_{\rho_\bullet})_\lambda)$, and the statement follows.
\end{proof}

\section{Application to the inverse Galois problem}\label{sec:IGP}

In this section we generalise results of~\cite{DiWi} by proving that compatible
systems with certain local properties and the global assumption that all (but a
density zero set) residual images are huge lead to the realisation of
$\PSp_{n}(\FF_{\ell^d})$ or $\PGSp_{n}(\FF_{\ell^d})$ as Galois groups over~$\QQ$
for certain fixed~$d$ and a positive density set of primes~$\ell$.
In Part~II~\cite{partII} we replace the residual huge image assumption by weaker local assumptions and in Part~III~\cite{partIII} we prove the existence
of such compatible systems.
Some of the ideas in this section are taken from~\cite{KLS1}, where the terminology
of $(n,p)$-groups is used.

\begin{defi}\label{defi:maximally-induced}
Let $\rho_\bullet = (\rho_\lambda)_\lambda$ be an $n$-dimensional
a.~e.\ symplectic compatible system as in Definition~\ref{defi:compatible}.
Let $\fq \in S$ be a place of~$F$ that is totally split over~$\QQ$ and does
not divide~$2n$. Denote by $q$ the rational prime under~$\fq$.
Let $\delta: G_{\QQ_{q^n}} \to L^\times$ be a tame symplectic character
of degree~$n$ and order $2p$ as in Section~3.1 of~\cite{KLS1} (where it is denoted~$\chi$)
for an odd prime~$p$ such that the order of $q$ in $\FF_p^\times$ is equal to~$n$
(note that, in particular, $n \mid (p-1)$).

We say that the compatible system $\rho_\bullet$ is {\em maximally
induced at~$\fq$ of order~$p$} if for each $\lambda$ not above~$q$ the restriction
of $\rho_\lambda$ to a decomposition group at~$\fq$ is equivalent to
$\Ind_{G_{\QQ_{q^n}}}^{G_{\QQ_q}} (\delta) \otimes \alpha$ (where we view $\delta$
as taking its values in $\overline{L}_\lambda$ via $L \hookrightarrow \overline{L}_\lambda$)
with an unramified character $\alpha: G_{\QQ_q} \to \overline{L}_\lambda^\times$.
\end{defi}

The definition is made precisely in such a way in order to ensure that
$\Ind_{G_{\QQ_{q^n}}}^{G_{\QQ_q}} (\delta)$ is symplectic and irreducible.
We refer to~\cite{KLS1} for more details (see also Part~II~\cite{partII}). Note moreover that the maximally induced place ensures that $\overline{\rho}_\lambda$ is absolutely
irreducible for almost all~$\lambda$. 
The existence of $\delta$ is a result of local class field theory.

\begin{lem}\label{lem:xi}
Let $p$ and $q$ be primes such that the order of $q$ in $\FF_p^\times$
is equal to~$n$.
Let $\zeta_p \in \Qbar$ be a primitive $p$-th root of unity.
Then the element $\xi_p := \sum_{i=0}^{n-1} \zeta_p^{q^i}$ has degree $\frac{p-1}{n}$ over~$\QQ$.
\end{lem}

\begin{proof}
See the paragraph above Proposition 2.16 of Chapter 2 of \cite{Washington}.
\end{proof}

Note that the element $\xi_p$ from Lemma~\ref{lem:xi} is the trace of an
element in the image of $\Ind_{G_{\QQ_{q^n}}}^{G_{\QQ_q}} (\delta)$.

\begin{lem}\label{lem:maximally-induced}
Let $\rho_\bullet = (\rho_\lambda)_\lambda$ be an $n$-dimensional a.~e.\ absolutely irreducible
a.~e.\ symplectic compatible system as in Definition~\ref{defi:compatible}.
Let $\psi:G_F \to L^\times$ be a character of finite order~$m$.
Assume that $\rho_\bullet$ has no complex multiplication and
that the multiplier of $\rho_\bullet$ is $\psi \chi_\ell^a$.
Assume also that there is a place $\fq \in S$
such that $\rho_\bullet$ is maximally induced at~$\fq$ of order~$p$.

Then $\xi_p := \sum_{i=0}^{q-1} \zeta_p^{q^i}$ is an element of $K_{\rho_\bullet}$.
\end{lem}

\begin{proof}
Let $\lambda$ be a place of~$L$ not above~$q$ (the prime below~$\fq$), such that $\rho_{\lambda}$ is absolutely irreducible and symplectic.
Let us take any $(\gamma,\epsilon)\in \calG_{[\rho_\lambda]}$
so that ${}^\gamma \rho_\lambda \sim \rho_\lambda \otimes \epsilon$.
We first want to prove that $\gamma(\xi_p)=\xi_p$.

There is an element $\sigma_0\in I_\fq$ (the inertia group at~$\fq$) such that $^{\gamma}\rho_{\lambda}(\sigma_0)$ and $(\rho_{\lambda}\otimes\epsilon)(\sigma_0)$ are conjugated to
\begin{equation}\label{eq:charpoly}\begin{pmatrix} \gamma(\zeta_p) & \ & \ & \ \\ \ & \gamma(\zeta_p)^q  \ & \ \\ \ & \ & \ddots & \ \\ \ & \ & \ & \gamma(\zeta_p)^{q^{n-1}} \end{pmatrix}\text{ and }
\epsilon(\sigma_0) \begin{pmatrix} \zeta_p & \ & \ & \ \\ \ & \zeta_p^q  \ & \ \\ \ & \ & \ddots & \ \\ \ & \ & \ & \zeta_p^{q^{n-1}} \end{pmatrix},
\end{equation}
respectively.
Since $\zeta_p\zeta_p^q\cdot \cdots \cdot \zeta_p^{q^n-1}=1$,  from the equality of the determinants of the matrices in \eqref{eq:charpoly} we get $\epsilon(\sigma_0)^{n}=1$.
But, $\epsilon(\sigma_0)$ is also a $p$-th root of unity (comparing the eigenvalues of the two matrices), whence $\epsilon(\sigma_0) = 1$, as $(n,p)=1$.
We then get $\gamma(\zeta_p)=\zeta_p^{q^i}$ for some $0\le i \le n-1$, and therefore $\gamma(\xi_p)=\zeta_p^{q^i} + \cdots + \zeta_p^{q^{i(n-1)}}=\zeta_p + \cdots + \zeta_p^{q^{n-1}}=\xi_p$.

Hence, we have established $\xi_p \in K_{[\rho_\lambda]}$.
But, since this holds for almost all~$\lambda$, we conclude from Theorem~\ref{thm:compatible} (using that, due to the presence of the maximally induced place, the residual representation $\overline{\rho}_{\lambda}$ is absolutely irreducible for almost all~$\lambda$)
that $\xi_p$ lies in~$K_{\rho_\bullet}$.
\end{proof}

Finally we will make use of the following proposition:

\begin{prop}[Prop. 7.2 of \cite{DiWi}]\label{DiWi7.2}
Let $L/\QQ$ be a finite field extension which contains a cyclic extension $M/\QQ$ of degree~$d$.
Then the set of primes $\ell$ such that there is an ideal $\lambda$ in $\cO_L$
of residue degree~$d$ has a positive density.
\end{prop}

The following is our main result in this part about the application to the inverse Galois problem.

\begin{thm}\label{thm:IGP}
Let $\rho_\bullet = (\rho_\lambda)_\lambda$ be an $n$-dimensional
a.~e.\ symplectic compatible system as in Definition~\ref{defi:compatible} with $K=\QQ$.
Let $\psi:G_F \to L^\times$ be a character of finite order.
Assume that the multiplier of $\rho_\bullet$ is $\psi \chi_\ell^a$.

Assume, moreover, that for all but a density zero set of places~$\lambda$ of~$L$
the residual representation $\overline{\rho}_\lambda$ has huge image.

Assume that for primes $\fL$ of~$F$ lying over the residue characteristic of~$\lambda$,
the representation $\overline{\rho}_\lambda$ satisfies Assumption~\ref{ass:shape}
(this is automatic if $F_\fL = \QQ_\ell$).
Moreover, assume that there is an integer~$k$, independent of~$\lambda$,
such that the numbers $a_{i,j}$ appearing in Assumption~\ref{ass:shape}
are bounded by~$k$.

Assume also that there is a place $\fq \in S$
such that $\rho_\bullet$ is maximally induced at~$\fq$ of order~$p$.

Then for any $d \mid \frac{p-1}{n}$ there exists a set~$\calL_d$ of rational primes~$\ell$
of positive density such that for all $\ell \in \calL_d$ there is a place $\lambda$ of~$L$
above~$\ell$ satisfying that the image of $\overline{\rho}_\lambda^\proj$
is $\PGSp_n(\FF_{\ell^d})$ or $\PSp_n(\FF_{\ell^d})$.
\end{thm}

\begin{proof}
Let $\xi_p := \sum_{i=0}^{q-1} \zeta_p^{q^i}$.
By Lemma~\ref{lem:xi} we know that $\xi_p$ has degree $\frac{p-1}{n}$ over~$\QQ$
and by Lemma~\ref{lem:maximally-induced} we also know that $\xi_p$ lies in $K_{\rho_\bullet}$.
Consequently, $K_{\rho_\bullet}$ contains a cyclic extension $M/\QQ$ of degree~$d$.
The theorem is now just a combination of Proposition~\ref{DiWi7.2}
and Corollary~\ref{cor:residual}.
\end{proof}

\bibliography{Bibliog}
\bibliographystyle{amsalpha}

\end{document}